\documentclass[11pt,reqno]{amsart}

\usepackage{amssymb}
\usepackage{amsmath,amsthm,amsfonts,amssymb,latexsym,mathrsfs,color,hyperref}
\usepackage{graphicx}
\usepackage{xspace}
\usepackage{multirow}
\usepackage{diagbox}
\usepackage{capt-of}
\usepackage{tikz}
\tikzset{
	level 1/.style = {sibling distance = 1.5cm},
	level 2/.style = {sibling distance = 0.8cm},
    level distance = 0.9 cm
}
\usetikzlibrary {decorations.pathmorphing}
\usetikzlibrary{matrix, positioning}
\tikzstyle{snakeline} = [decorate, decoration={snake, amplitude=.4mm, segment length=2mm}]

\usepackage{tikz-cd}
\usepackage[boxsize=1.3em, centerboxes]{ytableau}

\usepackage{tikz-qtree}
\tikzset{every tree node/.style={minimum width=0.1cm,draw,circle},
         blank/.style={draw=none},
         edge from parent/.style=
         {draw,edge from parent path={(\tikzparentnode) -- (\tikzchildnode)}},
         level distance=0.8cm}

\newtheorem{theorem}{Theorem}
\newtheorem{corollary}[theorem]{Corollary}
\newtheorem{proposition}[theorem]{Proposition}

\newtheorem{lemma}[theorem]{Lemma}

\newtheorem{example}[theorem]{Example}
\newtheorem{problem}[theorem]{Problem}

\newcommand{\men}{\mathcal{E}_n}

\newcommand{\drop}{{\rm drop\,}}

\newcommand{\altrun}{{\rm altrun\,}}

\newcommand{\run}{{\rm run\,}}

\newcommand{\fap}{{\rm fap\,}}

\newcommand{\ap}{{\rm ap\,}}

\newcommand{\des}{{\rm des\,}}

\newcommand{\exc}{{\rm exc\,}}

\newcommand{\fix}{{\rm fix\,}}

\newcommand{\msn}{\mathfrak{S}_n}

\newcommand{\ms}{\mathfrak{S}}

\newcommand{\as}{{\rm as\,}}

\newcommand{\lrf}[1]{\lfloor #1\rfloor}

\newcommand{\mbn}{{\mathcal B}_n}
\newcommand{\mqn}{{\mathcal Q}_n}

\DeclareMathOperator{\R}{\mathbb{R}}

\newcommand{\Stirling}[2]{\genfrac{\{}{\}}{0pt}{}{#1}{#2}}

\title{Determinantal representations of alternating run polynomials}
\author[S.-M.~Ma]{Shi-Mei Ma}
\address{School of Mathematics and Statistics, Shandong University of Technology, Zibo, Shandong 255000, P.R. China}
\email{shimeimapapers@163.com (S.-M. Ma)}
\author[H.~Bian]{Hong Bian}
\address{School of Mathematical Sciences, Xinjiang Normal University, Urumqi Xinjiang 830017, P.R. China}
\email{bh1218@163.com (H. Bian)}
\author[J.-Y.~Liu]{Jun-Ying Liu}
\address{School of Mathematics and Statistics, Shandong University of Technology, Zibo, Shandong 255000, P.R. China}
\email{jyliu6@163.com (J.-Y. Liu)}
\author{Jean Yeh}
\address{Department of Mathematics, National Kaohsiung Normal University, Kaohsiung 82444, Taiwan}
\email{chunchenyeh@nknu.edu.tw (J. Yeh)}
\author[Y.-N. Yeh]{Yeong-Nan Yeh}
\address{College of Mathematics and Physics, Wenzhou University, Wenzhou 325035, P.R. China}
\email{mayeh@math.sinica.edu.tw (Y.-N. Yeh)}
\begin{document}

\maketitle
\begin{abstract}
Based on a determinantal formula for the higher derivative of a quotient of two functions,
we first present the determinantal expressions of Eulerian polynomials and Andr\'e polynomials.
In particular, we discover that the Euler number (number of alternating permutations)
can be expressed as a lower Hessenberg determinant.
We then investigate the determinantal representations of the
up-down run polynomials and the types $A$ and $B$ alternating run polynomials. As applications,
we deduce several new recurrence relations, which imply the multiplicity of $-1$ in these
three kinds of polynomials. And then, we provide two determinantal representations for the alternating
run polynomials of dual Stirling permutations. In particular, we discover a close connection between
the alternating run polynomials of dual Stirling permutations and the type $B$ Eulerian polynomials.
\bigskip

\end{abstract}
\date{\today}
%
\section{Introduction}
Let $d_n$ be the {\it derangement number}, which counts
permutations on $[n]=\{1,2,\ldots,n\}$ such that leaves no elements in their original places.
The first few derangement numbers are $0,1,2,9,44,\ldots$.
The generating function of the derangement numbers (see~\cite{Brenti90} for instance) is given by
\begin{equation}\label{dn}
\sum_{n=0}^\infty d_n \frac{z^n}{n!}=\frac{\mathrm{e}^{-z}}{1-z}.
\end{equation}

In~\cite[Eq.~(11)]{Kit93}, it was given that for $n\geqslant 1$, one has
 \begin{equation*}
   d_{n+2}=\begin{vmatrix}
2&-1&0&\cdots&0&0&0\\
3&3&-1&\cdots&0&0&0\\
0&4&4&\cdots&0&0&0\\
\vdots&\vdots&\vdots&\ddots&\vdots&\vdots&\vdots\\
0&0&0&\cdots&n-1&-1&0\\
0&0&0&\cdots&n&n&-1\\
0&0&0&\cdots&0&n+1&n+1
    \end{vmatrix}_{n\times n}.
  \end{equation*}
Since then, various determinantal representations of $d_n$ have been
pursued by several authors, see~\cite{Chow2402,Vein99}.
The following result is stated in~\cite[p.~40]{Bourbaki04}. See~\cite[Lemma~6]{Chow2401} for a proof of it.
\begin{lemma}\label{lemma1}
Let $u=u(x)$ and $v=v(x)$ be two real functions which are $n$ times differentiable on an
interval $I\subset \R$. If one puts $\frac{\mathrm{d}^n}{\mathrm{d}x^n}\left(\frac{u}{v}\right)=(-1)^n\frac{w_n}{v^{n+1}}$ at every point where $v(x)\neq0$, then
 \begin{equation*}
  w_n=\begin{vmatrix}
u&v&0&0&\cdots&0&0&0\\
u'&v'&\binom{1}{1}v&0&\cdots&0&0&0\\
u''&v''&\binom{2}{1}v'&\binom{2}{2}v\cdots&0&0&0\\
\vdots&\vdots&\vdots&\vdots&\ddots&\vdots&\vdots&\vdots\\
u^{(n-1)}&v^{(n-1)}&\binom{n-1}{1}v^{(n-2)}&\binom{n-1}{2}v^{(n-3)}&\cdots&\cdots&\binom{n-1}{n-2}v{'}&\binom{n-1}{n-1}v\\
u^{(n)}&v^{(n)}&\binom{n}{1}v^{(n-1)}&\binom{n}{2}v^{(n-2)}&\cdots&\binom{n}{n-3}v^{(3)}&\binom{n}{n-2}v{''}&\binom{n}{n-1}v{'}
    \end{vmatrix}_{(n+1)\times (n+1)}.
  \end{equation*}
\end{lemma}
By Lemma~\ref{lemma1}, Munarini~\cite{Munarini20} showed that the derangement
polynomial $d_n(q)$ (major polynomial over the set of derangements on $[n]$)
can be expressed as the determinant of either an $n\times n$ tridiagonal matrix or an $n\times n$ lower Hessenberg matrix,
Qi et al.~\cite{Qi15,Qi16,Qi21} deduced determinantal representations of derangement numbers, Eulerian polynomials, Bernoulli polynomials and derivative polynomials,
Chow~\cite{Chow23,Chow2401,Chow2402}
deduced determinantal representations of $q$-derangement numbers and Eulerian polynomials of types $B$ and $D$.
In particular, using~\eqref{dn}, Qi-Wang-Guo~\cite[Theorem~1]{Qi16} found that $d_n$ can be expressed as
a tridiagonal determinant:
 \begin{equation*}
   d_{n}=\begin{vmatrix}
-1&1&0&0&0&\cdots&0&0&0\\
0&0&1&0&0&\cdots&0&0&0\\
0&-1&1&1&0&\cdots&0&0&0\\
0&0&-2&2&1&\cdots&0&0&0\\
0&0&0&-3&3&\cdots&0&0&0\\
\vdots&\vdots&\vdots&\vdots&\vdots&\ddots&\vdots&\vdots&\vdots\\
0&0&0&0&0&\cdots&n-3&1&0\\
0&0&0&0&0&\cdots&-(n-2)&n-2&1\\
0&0&0&0&0&\cdots&0&-(n-1)&n-1
    \end{vmatrix}=\begin{vmatrix}e_{ij}\end{vmatrix}_{(n+1)\times (n+1)},
  \end{equation*}
where $$e_{ij}=\left\{
                 \begin{array}{ll}
                   1, & \hbox{if $i-j=-1$;} \\
                   i-2, & \hbox{if $i-j=0$;} \\
                   2-i, & \hbox{if $i-j=1$;} \\
                   0, & \hbox{if $i-j\neq 0,\pm 1$.}
                 \end{array}
               \right.
$$

Let $\msn$ denote the symmetric group of all permutations of $[n]:=\{1,2,3,\ldots,n\}$. For any $\pi\in\msn$, written
as the word $\pi(1)\pi(2)\cdots\pi(n)$, the entry $\pi(i)$ is called a {\it fixed point} if $i\in [n]$ and $\pi(i)=i$.
Then $d_n=\#\{\pi\in\msn: \fix(\pi)=0\}$.
For any $\pi\in\msn$, the entry $\pi(i)$ is called
\begin{itemize}
  \item a {\it descent} if $i\in [n-1]$ and $\pi(i)>\pi(i+1)$;
  \item an {\it excedance} if $i\in [n-1]$ and $\pi(i)>i$;
    \item a {\it drop} if $i\in \{2,3,\ldots,n\}$ and $\pi(i)<i$.
\end{itemize}
Let $\des(\pi)$ (resp.~$\exc(\pi)$, $\drop(\pi)$) be the number of descents (resp.~excedances, drops) of $\pi$.
It is well known that descents, excedances and drops are equidistributed over $\msn$, and their common enumerative polynomial
is the {\it Eulerian polynomial} (see~\cite[A008292]{Sloane}):
$$A_n(x)=\sum_{\pi\in\msn}x^{\des(\pi)+1}=\sum_{\pi\in\msn}x^{\exc(\pi)+1}=\sum_{\pi\in\msn}x^{\drop(\pi)+1}.$$
The exponential generating function of $A_n(x)$ is given as follows:
\begin{equation}\label{Axz-EGF}
A(x;z)=\sum_{n=0}^{\infty}A_n(x)\frac{z^n}{n!}=\frac{1-x}{1-x\mathrm{e}^{z(1-x)}}.
\end{equation}

Let $\{f_n(x)\}$ be a sequences of polynomials, and let $f(x;z)=\sum_{n\geqslant 0}f_n(x)\frac{z^n}{n!}$.
If $f(x;z)$ can be written as a quotient
\begin{equation*}\label{uv}
f(x;z)=\frac{u(x;z)}{v(x;z)},
\end{equation*}
then using Lemma~\ref{lemma1}, one can derive that
$f_n(x)$ can be expressed as a determinant of order $n+1$, since $f_n(x)=\lim_{z\rightarrow 0}\frac{\partial^n}{\partial z^n}f(x;z)$.
It should be noted that both $u(x;z)$ and $v(x;z)$ have simple expressions in the most of the aforementioned studies, see~\cite{Chow2401,Chow2402} for instances.
For example,
let $u(x;z)=1-x$ and $v(x,z)=1-x\mathrm{e}^{z(1-x)}$. Then for $k\geqslant 1$,
$\frac{\partial^k}{\partial z^k} u(x;z)=0$ and $\frac{\partial^k}{\partial z^k} v(x;z)=-x(1-x)^k\mathrm{e}^{z(1-x)}$.
Applying Lemma~\ref{lemma1}, and letting $z\rightarrow 0$, Chow~\cite{Chow23} found that $A_n(x)$ can be expressed as the following lower
Hessenberg determinant:
 \begin{equation}\label{Anx-det}
  A_{n}(x)=\begin{vmatrix}
1&1&0&\cdots&0&0\\
0&x&-1&\cdots&0&0\\
0&x(1-x)&2x&\cdots&0&0\\
\vdots&\vdots&\vdots&&\ddots&\vdots\\
0&x(1-x)^{n-2}&\binom{n-1}{1}x(1-x)^{n-3}&\cdots&\binom{n-1}{n-2}x&-1\\
0&x(1-x)^{n-1}&\binom{n}{1}x(1-x)^{n-2}&\cdots&\binom{n}{n-2}x(1-x)&\binom{n}{n-1}x
    \end{vmatrix}_{(n+1)\times(n+1)}.
  \end{equation}

This paper is motivated by the following problem.
\begin{problem}
Assume that
$f(x;z)=\sum_{n\geqslant 0}f_n(x)\frac{z^n}{n!}=\frac{u(x;z)}{v(x;z)}$. If both $u(x;z)$ and $v(x;z)$
have complicated expressions, or we can not deduce explicit formulas for $u(x;z)$ and $v(x;z)$,
how to deduce a determinantal formula of $f_n(x)$?
\end{problem}

For example, the type $A$ alternating run polynomials can be defined by
\begin{equation}\label{RnxAnx}
R_n(x)=\left(\frac{1+x}{2}\right)^{n-1}(1+w)^{n+1}A_n\left(\frac{1-w}{1+w}\right)
\end{equation}
for $n\geqslant 2$ (see David-Barton~\cite[157-162]{David62} and~\cite{Ma20}), where $w=\sqrt{\frac{1-x}{1+x}}$.
By the method of characteristics,
Carlitz~\cite{Carlitz78} proved that
\begin{equation}\label{CarlitzGF}
R(x;z)=\sum_{n=0}^\infty (1-x^2)^{-n/2}\frac{z^n}{n!}\sum_{k=0}^nR(n+1,k)x^{n-k}=\frac{1-x}{1+x}\left(\frac{\sqrt{1-x^2}+\sin z}{x- \cos z}\right)^2.
\end{equation}

The original purpose of this paper is to derive determinantal representation of $R_n(x)$ by using Lemma~\ref{lemma1}.
Not like~\eqref{dn} or~\eqref{Axz-EGF}, the exponential generating function~\eqref{CarlitzGF} seems complicated.
In the next section, we collect the definitions, notations and preliminary
results that will be used in the rest of this work. In particular, we discover that the Euler number $E_{n+1}$ (number of alternating permutations in $\msn$)
can be expressed as a lower Hessenberg determinant of order $n$.
In Section~\ref{section3}, we investigate the up-down run polynomials and
the types $A$ and $B$ alternating run polynomials.
For example, in Theorem~\ref{thmsntn2}, we find that
the type $A$ alternating run $R_{n+1}(x)$ can be expressed as a lower Hessenberg determinant of order $n$:
 \begin{equation*}
\begin{vmatrix}
2x&-1&0&0&\cdots&0&0\\
2xR_2(x)&f_1(x)&-\binom{1}{1}&0&\cdots&0&0\\
2xR_3(x)&f_2(x)&\binom{2}{1}f_1(x)&-\binom{2}{2}&\cdots&0&0\\
\vdots&\vdots&\vdots&\vdots&\ddots&\vdots&\vdots\\
2xR_{n-1}(x)&{f}_{n-2}(x)&\binom{n-2}{1}{f}_{n-3}(x)&\binom{n-2}{2}{f}_{n-4}(x)&\cdots&\binom{n-2}{n-3}f_1(x)&-\binom{n-2}{n-2}\\
2xR_{n}(x)&{f}_{n-1}(x)&\binom{n-1}{1}{f}_{n-2}(x)&\binom{n-1}{2}{f}_{n-3}(x)&\cdots&\binom{n-1}{n-3}f_2(x)&\binom{n-1}{n-2}f_1(x)
    \end{vmatrix},
  \end{equation*}
where $f_n(x)=(-1)^{n+1}(1-x^2)^{\lrf{n/2}}$. In the case when $n=3$, we see that
$$R_4(x)=\begin{vmatrix}
2x&-1&0\\
4x^2&1&-1\\
4x^2+8x^3&-1+x^2&2
    \end{vmatrix}=2 x + 12 x^2 + 10 x^3.$$
In Section~\ref{dualStirling}, we discuss the alternating run polynomials of dual Stirling permutations.
\section{Preliminaries}
\subsection{Lower Hessenberg determinants and context-free grammars}
\hspace*{\parindent}

The $n\times n$ {\it lower Hessenberg matrix} $H_n$ is defined as follows:
\begin{equation*}
H_n=\begin{pmatrix}
h_{11}&h_{12}&0&\cdots&0&0\\
h_{21}&h_{22}&h_{23}&\cdots&0&0\\
h_{31}&h_{32}&h_{33}&\cdots&0&0\\
\vdots&\vdots&\vdots&\ddots&\vdots&\vdots\\
h_{n-1,1}&h_{n-1,2}&h_{n-1,3}&\cdots&h_{n-1,n-1}&h_{n-1,n}\\
h_{n,1}&h_{n,2}&h_{n,3}&\cdots&h_{n,n-1}&h_{n,n}\\
    \end{pmatrix},
  \end{equation*}
where $h_{ij}=0$ if $j>i+1$.
Hessenberg matrices are one of the important matrices in numerical analysis~\cite{Cahill02,Elouafi09,Kilic17}.
Setting $H_0=1$,
Cahill et al.~\cite{Cahill02} gave a recursion for
the determinant of the matrix $H_n$ as follows:
\begin{equation}\label{Hn-recu}
\det H_{n+1}=h_{n+1,n+1}\det H_{n}+\sum_{r=1}^{n}\left((-1)^{n+1-r}h_{n+1,r}\prod_{j=r}^{n}h_{j,j+1}\det H_{r-1}\right).
\end{equation}

Following Chen~\cite{Chen93}, a {\it context-free grammar} $G$ over an alphabet
$V$ is defined as a set of substitution rules replacing a letter in $V$ by a formal function over $V$.
The formal derivative $D_G$ with respect to $G$ satisfies the derivation rules:
$$D_G(u+v)=D_G(u)+D_G(v),~D_G(uv)=D_G(u)v+uD_G(v).$$
So the {\it Leibniz rule} holds:
\begin{equation}\label{Leibniz}
D_G^n(uv)=\sum_{k=0}^n\binom{n}{k}D_G^k(u)D_G^{n-k}(v).
\end{equation}
Recently, context-free grammars have been used extensively in the study of permutations, perfect matchings and increasing trees,
see~\cite{Chen2301,Hao05,Ma23} for instances.

In this paper, we always let $D_G$ be the formal derivative associated with the grammar $G$. As an illustration,
we now recall a classical result in this topic.
\begin{proposition}[{\cite{Dumont96}}]\label{grammar01}
If $G=\{a\rightarrow ab, b\rightarrow ab\}$, then $D_{G}^n(a)=a^{n+1}A_n\left(\frac{b}{a}\right)$. Thus a grammatical description of the Eulerian polynomial is given as follows: $$D_{G}^n(a)|_{a=1,b=x}=A_n(x).$$
\end{proposition}

We now give the grammatical version of Lemma~\ref{lemma1}.
\begin{lemma}\label{lemma2}
Let $G=\{a\rightarrow f_1(a,b,\ldots),~b\rightarrow f_2(a,b,\ldots),\ldots\}$ be a given grammar. For formal functions $u(a,b,\ldots)$ and $v(a,b,\ldots)$, we have $D_G^n\left(\frac{u}{v}\right)=(-1)^n\frac{w_n}{v^{n+1}}$, where
$w_n$ is the following lower
Hessenberg determinant of order $n+1$:
 \begin{equation*}
  w_n=\begin{vmatrix}
u&v&0&0&\cdots&0&0\\
D_G(u)&D_G(v)&\binom{1}{1}v&0&\cdots&0&0\\
D_G^2(u)&D_G^2(v)&\binom{2}{1}D_G(v)&\binom{2}{2}v&\cdots&0&0\\
\vdots&\vdots&\vdots&\vdots&\ddots&\vdots&\vdots\\
D_G^{n-1}(u)&D_G^{n-1}(v)&\binom{n-1}{1}D_G^{n-2}(v)&\binom{n-1}{2}D_G^{n-3}(v)&\cdots&\binom{n-1}{n-2}D_G(v)&\binom{n-1}{n-1}v\\
D_G^{n}(u)&D_G^{n}(v)&\binom{n}{1}D_G^{n-1}(v)&\binom{n}{2}D_G^{n-2}(v)&\cdots&\binom{n}{n-2}D_G^{2}(v)&\binom{n}{n-1}D_G(v)
    \end{vmatrix}.
  \end{equation*}
\end{lemma}
\begin{proof}
Set $w=\frac{u}{v}$. Then $u=wv$. By~\eqref{Leibniz}, we see that
$$D_G^n(u)=\sum_{k=0}^n\binom{n}{k}D_G^{n-k}(v)D_G^k(w),$$
when can be expressed as the following matrix equation:
$$\begin{pmatrix}
u\\
D_G(u)\\
D_G^2(u)\\
D_G^3(u)\\
\vdots\\
D_{G}^{n-1}(u)\\
D_G^n(u)
    \end{pmatrix}=\begin{pmatrix}
v&0&0&\cdots&0\\
D_G(v)&\binom{1}{1}v&0&\cdots&0\\
D_G^2(v)&\binom{2}{1}D_G(v)&\binom{2}{2}v&\cdots&0\\
D_G^3(v)&\binom{3}{1}D_G^2(v)&\binom{3}{2}D_G(v)&\cdots&0\\
\vdots&\vdots&\vdots&\ddots&\vdots\\
D_{G}^{n-1}(v)&\binom{n-1}{1}D_G^{n-2}(v)&\binom{n-1}{2}D_G^{n-3}(v)&\cdots&0\\
D_G^n(v)&\binom{n}{1}D_G^{n-1}(v)&\binom{n}{2}D_G^{n-2}(v)&\cdots&\binom{n}{n}v
    \end{pmatrix}\begin{pmatrix}
w\\
D_G(w)\\
D_G^2(w)\\
D_G^3(w)\\
\vdots\\
D_{G}^{n-1}(w)\\
D_G^n(w)
    \end{pmatrix}.$$
Note that the coefficient matrix is triangular with determinant $v^{n+1}$, by Cramer's rule, we get
$$D_G^n(w)=\frac{1}{v^{n+1}}\begin{vmatrix}
v&0&0&\cdots&0&u\\
D_G(v)&\binom{1}{1}v&0&\cdots&0&D_G(u)\\
D_G^2(v)&\binom{2}{1}D_G(v)&\binom{2}{2}v&\cdots&0&D_G^2(u)\\
D_G^3(v)&\binom{3}{1}D_G^2(v)&\binom{3}{2}D_G(v)&\cdots&0&D_G^3(u)\\
\vdots&\vdots&\vdots&\ddots&\vdots&\vdots\\
D_{G}^{n-1}(v)&\binom{n-1}{1}D_G^{n-2}(v)&\binom{n-1}{2}D_G^{n-3}(v)&\cdots&\binom{n-1}{n-1}v&D_G^{n-1}(u)\\
D_G^n(v)&\binom{n}{1}D_G^{n-1}(v)&\binom{n}{2}D_G^{n-2}(v)&\cdots&\binom{n}{n-1}D_G(v)&D_G^n(u)
    \end{vmatrix}.$$
Interchanging of adjacent columns such that the last column becomes the first column,
we must introduce a minus sign upon each swapping of a pair of adjacent columns, so
we get the sign $(-1)^n$, as desired. This completes the proof.
\end{proof}
In the next subsection, we give some illustrations of Lemma~\ref{lemma2}.
\subsection{Eulerian polynomials and Andr\'e polynomials}
\hspace*{\parindent}

Consider Proposition~\ref{grammar01}.
Let $G=\{a\rightarrow ab, b\rightarrow ab\}$. Note that $D_{G}^n(a)=D_{G}^n(\frac{u}{v})$, where
$u(a,b)=1$ and $v(a,b)={1}/{a}$. For $k\geqslant 1$, we have $$D_G^k\left(u\right)=0,~D_G^k\left(v\right)=-\frac{b}{a}(a-b)^{k-1}.$$
When $a=1$ and $b=x$, combining Lemma~\ref{lemma2} and Proposition~\ref{grammar01},
we see that
$A_n(x)=(-1)^nw_n$, where $w_n$ the
the following lower
Hessenberg determinant of order $n+1$:
 \begin{equation*}
\begin{vmatrix}
1&1&0&0&\cdots&0&0\\
0&-x&\binom{1}{1}&0&\cdots&0&0\\
0&-x(1-x)&-\binom{2}{1}x&\binom{2}{2}&\cdots&0&0\\
\vdots&\vdots&\vdots&\vdots&\ddots&\vdots&\vdots\\
0&-x(1-x)^{n-2}&-\binom{n-1}{1}x(1-x)^{n-3}&-\binom{n-1}{2}x(1-x)^{n-4}&\cdots&-\binom{n-1}{n-2}x&\binom{n-1}{n-1}\\
0&-x(1-x)^{n-1}&-\binom{n}{1}x(1-x)^{n-2}&-\binom{n}{2}x(1-x)^{n-3}&\cdots&-\binom{n}{n-2}x(1-x)&-\binom{n}{n-1}x
    \end{vmatrix}.
 \end{equation*}
Except for the first column, multiplying each of the other columns of $w_n$ by $(-1)$,
we get~\eqref{Anx-det}.
If we set $u(a,b)=b$ and $v(a,b)={1}/{a}$,
then for $n\geqslant 1$, we have
$$D_{G}^{n+1}(a)=D_{G}^n(ab)=D_{G}^n\left(\frac{u}{v}\right),~D_{G}^n(u)=D_{G}^n(b)=a^{n+1}A_n\left(\frac{b}{a}\right),~D_G^n\left(v\right)=-\frac{b}{a}(a-b)^{n-1}.$$
Using Lemma~\ref{lemma2} and Proposition~\ref{grammar01}, letting $a=1$ and $b=x$, we find that $A_{n+1}(x)$ is expressible as $(-1)^nw_n$, where $w_n$ is the following lower
Hessenberg determinant of order $n+1$:
 \begin{equation*}
\begin{vmatrix}
x&1&0&\cdots&0&0\\
A_1(x)&-x&1&\cdots&0&0\\
A_2(x)&-x(1-x)&-2x&\cdots&0&0\\
\vdots&\vdots&\vdots&&\ddots&\vdots\\
A_{n-1}(x)&-x(1-x)^{n-2}&-\binom{n-1}{1}x(1-x)^{n-3}&\cdots&-\binom{n-1}{n-2}x&1\\
A_n(x)&-x(1-x)^{n-1}&-\binom{n}{1}x(1-x)^{n-2}&\cdots&-\binom{n}{n-2}x(1-x)&-\binom{n}{n-1}x
    \end{vmatrix}.
  \end{equation*}
Except for the first column of $w_n$, multiplying each of the other columns by $(-1)$, we get the following result.
\begin{theorem}\label{Eulerian-det}
For any $n\geqslant 0$, the Eulerian polynomial $A_{n+1}(x)$ can be expressed as the following lower
Hessenberg determinant of order $n+1$:
 \begin{equation*}
A_{n+1}(x)=\begin{vmatrix}
x&-1&0&\cdots&0&0\\
A_1(x)&x&-1&\cdots&0&0\\
A_2(x)&x(1-x)&2x&\cdots&0&0\\
\vdots&\vdots&\vdots&&\ddots&\vdots\\
A_{n-1}(x)&x(1-x)^{n-2}&\binom{n-1}{1}x(1-x)^{n-3}&\cdots&\binom{n-1}{n-2}x&-1\\
A_n(x)&x(1-x)^{n-1}&\binom{n}{1}x(1-x)^{n-2}&\cdots&\binom{n}{n-2}x(1-x)&\binom{n}{n-1}x
    \end{vmatrix}.
  \end{equation*}
\end{theorem}

\begin{example}
When $n=3$, one has
 \begin{equation*}
  A_{4}(x)=\begin{vmatrix}
x&-1&0&0\\
x&x&-1&0\\
x+x^2&x(1-x)&2x&-1\\
x+4x^2+x^3&x(1-x)^2&3x(1-x)&3x\\
    \end{vmatrix}=x+11x^2+11x^3+x^4.
  \end{equation*}
\end{example}

In the case of Theorem~\ref{Eulerian-det},
it follows from~\eqref{Hn-recu} that $\det H_{n+1}=A_{n+1}(x),~h_{n+1,1}=A_n(x)$, $h_{n+1,n+1}=\binom{n}{n-1}x=nx$, $\prod_{j=r}^{n}h_{j,j+1}=(-1)^{n-r+1}$ for $1\leqslant r\leqslant n$, and
$$h_{n+1,r}=\binom{n}{r-2}x(1-x)^{n-r+1},$$
where $2\leqslant r\leqslant n$. So we get the following result.
\begin{corollary}
For $n\geqslant 1$, we have $A_1(x)=x$ and
\begin{align*}
A_{n+1}(x)&=(1+nx)A_n(x)+x\sum_{r=2}^n\binom{n}{r-2}(1-x)^{n-r+1}A_{r-1}(x)\\
&=A_n(x)+x\sum_{r=2}^{n+1}\binom{n}{r-2}(1-x)^{n-r+1}A_{r-1}(x).
\end{align*}
\end{corollary}

As another illustration of Lemma~\ref{lemma2}, we shall consider the famous Andr\'e polynomials.
An increasing tree on $[n]$ is a rooted tree with vertex set $\{0,1,2,\ldots,n\}$ in which the labels
of the vertices are increasing along any path from the root $0$. The degree of a vertex is meant to be the number of its children.
A {\it 0-1-2 increasing tree} is an increasing tree in which the degree of any vertex is
at most two.  Given a 0-1-2 increasing tree $T$, let $\ell(T)$ denote the number
of leaves of $T$, and $u(T)$ denote the number of vertices of $T$ with degree $1$. The
{\it Andr\'e polynomials} are defined by
$$E_n(x,y)=\sum_{T\in \men}x^{\ell(T)}y^{u(T)},$$
where $\men$ is the set of 0-1-2 increasing trees on $[n-1]$.
Here are the first few values:
\begin{equation*}
E_1(x,y)=x,~
E_2(x,y)=xy,~
E_3(x,y)=xy^2+x^2,~
E_4(x,y)=xy^3+4x^2y.
\end{equation*}

Using a grammatical labeling of 0-1-2 increasing trees,
Dumont~\cite{Dumont96} discovered that
\begin{equation}\label{DGNy}
D_G^n(y)=E_n(x,y),
\end{equation}
where $G=\{x\rightarrow xy,~y\rightarrow x\}$.
Let $E(x,y;z)=\sum_{n=0}^\infty E_n(x,y)\frac{z^n}{n!}$.
By solving a differential equation, Foata-Sch\"utzenberger~\cite{Foata73} found that
\begin{equation}\label{foata}
\begin{small}
E(x,y;z)=\frac{x\sqrt{2x-y^2}+y(2x-y^2)\sin(z\sqrt{2x-y^2})-(x-y^2)\sqrt{2x-y^2}\cos(z\sqrt{2x-y^2})}{(x-y^2)\sin(z\sqrt{2x-y^2})+y\sqrt{2x-y^2}\cos(z\sqrt{2x-y^2})}.
\end{small}
\end{equation}

We say that $\pi\in\msn$ is {\it alternating}, if $\pi(1)>\pi(2)<\pi(3)>\cdots \pi(n)$, i.e., $\pi(2i-1)>\pi(2i)$ and $\pi(2i)<\pi(2i+1)$ for all $i$.
The classical {\it Euler number} $E_n$ counts
alternating permutations in $\msn$, see~\cite[A000111]{Sloane}. In particular, $E_3=\#\{213,312\}=2$.
Dumont~\cite{Dumont96} noted that
\begin{equation}\label{En}
E(1,1;z)=\tan z+ \sec z=\sum_{n=0}^\infty E_n\frac{z^n}{n!}.
\end{equation}
Thus $E_n=\#\men=E_n(1,1)$.
Following~\cite[Theorem~11]{Ma132}, a close connection between Andr\'e polynomials and the type $A$ alternating run polynomials is given as follows:
$$R_n(x)=2(1+x)^{n-1}E_n\left(\frac{x}{1+x},1\right)~\text{for $n\geqslant 2$}.$$

We now give a dual formula of~\eqref{foata}.
\begin{theorem}\label{Andre}
For any $n\geqslant 1$, we define
$${e}_n(x,y)=\left\{
           \begin{array}{ll}
             (-1)^{k-1}y(2x-y^2)^{k-1}, & \hbox{if $n=2k-1$;} \\
             (-1)^{k-1}{(x-y^2)}(2x-y^2)^{k-1}, & \hbox{if $n=2k$.}
           \end{array}
         \right.
$$
In particular, $e_1(x,y)=y,~e_2(x,y)=x-y^2$ and $e_3(x,y)=-y(2x-y^2)$.
The And\'re polynomial $E_{n+1}(x,y)$ can be expressed as the Hessenberg determinant of order $n$:
 \begin{equation*}
 \begin{vmatrix}
x{e}_1(x,y)&-\binom{1}{1}&0&\cdots&0&0\\
x{e}_2(x,y)&\binom{2}{1}{e}_1(x,y)&-\binom{2}{2}&\cdots&0&0\\
\vdots&\vdots&\vdots&\ddots&\vdots&\vdots\\
x{e}_{n-1}(x,y)&\binom{n-1}{1}{e}_{n-2}(x,y)&\binom{n-1}{2}{e}_{n-3}(x,y)&\cdots&\binom{n-1}{n-2}{e}_1(x,y)&-\binom{n-1}{n-1}\\
x{e}_n(x,y)&\binom{n}{1}{e}_{n-1}(x,y)&\binom{n}{2}{e}_{n-2}(x,y)&\cdots&\binom{n}{n-2}{e}_2(x,y)&\binom{n}{n-1}{e}_1(x,y)
    \end{vmatrix}.
  \end{equation*}
\end{theorem}
\begin{proof}
Let $G=\{x\rightarrow xy,~y\rightarrow x\}$. It follows from~\eqref{DGNy} that
$$E_{n+1}(x,y)=D_G^{n+1}(y)=D_G^n(x)=D_G^n\left(\frac{1}{\frac{1}{x}}\right).$$
Put $u=1$ and $v=\frac{1}{x}$.
Then $D_G^k(u)=0$ for $k\geqslant 1$. Note that
$$D_G(v)=D_G\left(\frac{1}{x}\right)=-\frac{y}{x},~D_G^2(v)=D_G\left(-\frac{y}{x}\right)=\frac{y^2-x}{x},$$
$$D_G^3(v)=D_G^3\left(\frac{y^2-x}{x}\right)=\frac{y}{x}(2x-y^2).$$
Since $D_G(2x-y^2)=0$ and $D_G\left(\frac{y}{x}\right)=\frac{x-y^2}{x}$, we find that
$$D_G^4(v)=\frac{x-y^2}{x}(2x-y^2),~D_G^5(v)=-\frac{y}{x}(2x-y^2)^2,$$
$$D_G^6(v)=-\frac{x-y^2}{x}(2x-y^2)^2,~D_G^7(v)=\frac{y}{x}(2x-y^2)^3.$$
By induction, it is easy to verify that for $k\geqslant 1$, we have
$$D_G^{2k-1}(v)=(-1)^k\frac{y}{x}(2x-y^2)^{k-1},~D_G^{2k}(v)=(-1)^k\frac{x-y^2}{x}(2x-y^2)^{k-1}.$$
We define
$$\widetilde{e}_n(x,y)=\left\{
           \begin{array}{ll}
             (-1)^k\frac{y}{x}(2x-y^2)^{k-1}, & \hbox{if $n=2k-1$;} \\
             (-1)^k\frac{x-y^2}{x}(2x-y^2)^{k-1}, & \hbox{if $n=2k$.}
           \end{array}
         \right.
$$
Using Lemma~\ref{lemma2}, we obtain that $E_{n+1}(x,y)=(-1)^nx^{n+1}w_n$, where $w_n$ is the lower Hessenberg determinant of order $n+1$:
 \begin{equation*}
 \begin{vmatrix}
1&\frac{1}{x}&0&0&\cdots&0&0\\
0&-\frac{y}{x}&\binom{1}{1}\frac{1}{x}&0&\cdots&0&0\\
0&\widetilde{e}_2(x,y)&-\binom{2}{1}\frac{y}{x}&\binom{2}{2}\frac{1}{x}&\cdots&0&0\\
\vdots&\vdots&\vdots&\vdots&\ddots&\vdots&\vdots\\
0&\widetilde{e}_{n-1}(x,y)&\binom{n-1}{1}\widetilde{e}_{n-2}(x,y)&\binom{n-1}{2}\widetilde{e}_{n-3}(x,y)&\cdots&-\binom{n-1}{n-2}\frac{y}{x}&\binom{n-1}{n-1}\frac{1}{x}\\
0&\widetilde{e}_n(x,y)&\binom{n}{1}\widetilde{e}_{n-1}(x,y)&\binom{n}{2}\widetilde{e}_{n-2}(x,y)&\cdots&\binom{n}{n-2}\widetilde{e}_2(x,y)&-\binom{n}{n-1}\frac{y}{x}
    \end{vmatrix}.
  \end{equation*}
Expanding $w_n$ along the first column, multiplying the first column of the resulting determinant by $-x^2$ and multiplying each of the other columns by $-x$, we get
 \begin{equation*}
 \begin{vmatrix}
xy&-\binom{1}{1}&0&\cdots&0&0\\
-x^2\widetilde{e}_2(x,y)&\binom{2}{1}y&-\binom{2}{2}&\cdots&0&0\\
\vdots&\vdots&\vdots&\ddots&\vdots&\vdots\\
-x^2\widetilde{e}_{n-1}(x,y)&-x\binom{n-1}{1}\widetilde{e}_{n-2}(x,y)&-x\binom{n-1}{2}\widetilde{e}_{n-3}(x,y)&\cdots&\binom{n-1}{n-2}y&-\binom{n-1}{n-1}\\
-x^2\widetilde{e}_n(x,y)&-x\binom{n}{1}\widetilde{e}_{n-1}(x,y)&-x\binom{n}{2}\widetilde{e}_{n-2}(x,y)&\cdots&-x\binom{n}{n-2}\widetilde{e}_2(x,y)&\binom{n}{n-1}y
    \end{vmatrix},
  \end{equation*}
which yields the desired result.
\end{proof}

When $x=y=1$, combining Theorem~\ref{Andre} and~\eqref{Hn-recu}, we arrive at the following result.
\begin{corollary}
For any $n\geqslant 1$, let $a_{2n-1}=(-1)^{n-1}$ and $a_{2n}=0$.
The Euler number $E_{n+1}$ can be expressed as the Hessenberg determinant of order $n$:
 \begin{equation*}
 \begin{vmatrix}
a_1&-\binom{1}{1}&0&\cdots&0&0\\
a_2&\binom{2}{1}a_1&-\binom{2}{2}&\cdots&0&0\\
a_3&\binom{3}{1}a_2&\binom{3}{2}a_1&\cdots&0&0\\
\vdots&\vdots&\vdots&\ddots&\vdots&\vdots\\
a_{n-1}&\binom{n-1}{1}{a}_{n-2}&\binom{n-1}{2}{a}_{n-3}&\cdots&\binom{n-1}{n-2}a_1&-\binom{n-1}{n-1}\\
a_n&\binom{n}{1}{a}_{n-1}&\binom{n}{2}{a}_{n-2}&\cdots&\binom{n}{n-2}a_2&\binom{n}{n-1}a_1
    \end{vmatrix}.
  \end{equation*}
Moreover, for $n\geqslant 1$, the Euler numbers $E_n$ satisfy the recurrence relation
$$E_{n+1}=\sum_{r=1}^n\binom{n}{r-1}a_{n-r+1}E_r,$$
with the initial values $E_1=E_2=1$ and $E_3=2$.
\end{corollary}
\subsection{Alternating run and up-down run polynomials}
\hspace*{\parindent}

For each $\pi=\pi(1)\pi(2)\cdots\pi(n)\in\msn$, a {\it run} of $\pi$ is a maximal segment (a subsequence
consisting of consecutive elements) that is either increasing or decreasing.
For example, the alternating runs of $\pi=82315467$ are $82,23,31,15,54,467$.
Let $\run(\pi)$ denote the number of alternating runs of $\pi$.
The {\it type $A$ alternating run polynomials} are defined by
$$R_n(x)=\sum_{\pi\in\msn}x^{\run(\pi)}=\sum_{k=0}^{n-1}R(n,k)x^k.$$
Andr\'e~\cite{Andre84} found that
\begin{equation}\label{Rnk-recurrence}
R(n,k)=kR(n-1,k)+2R(n-1,k-1)+(n-k)R(n-1,k-2)
\end{equation}
for $n>k\geqslant 1$, where $R(1,0)=1$ and $R(1,k)=0$ for $k\geqslant 1$.

Using column generating functions (see~\cite[Section~4]{CW08}), Canfield-Wilf
obtained that
$$R(n,k)=\frac{1}{2^{k-2}}k^n-\frac{1}{2^{k-4}}(k-1)^n+\psi_2(n,k)(k-2)^n+\cdots+\psi_{k-1}(n,k)$$
for $n\geqslant 2$, in which each $\psi_i(n,k)$ is a polynomial in $n$ whose degree in $n$ is $\lrf{i/2}$.
By solving partial differential equations, Stanley~\cite{Sta08} discovered that
\begin{equation*}\label{Rnk-Stan}
R(n,k)=\sum_{i=0}^k\frac{1}{2^{i-1}}(-1)^{k-i}z_{k-i}\sum_{\substack{r+2m\leqslant i\\r\equiv i\bmod 2}}(-2)^m\binom{i-m}{(i+r)/2}\binom{n}{m}r^n,
\end{equation*}
where $z_0=2$ and $z_n=4$ for $n\geqslant 1$. Ma~\cite{Ma12} found another explicit formula by
expressing $R_n(x)$ in terms of the derivative polynomials of tangent function.
The reader is referred to~\cite{Ma20,Zhuang16} for the recent work on alternating run polynomials.

An {\it alternating subsequence} of $\pi=\pi(1)\pi(2)\cdots\pi(n)$ is a subsequence $\pi({i_1})\cdots \pi({i_k})$ satisfying
$$\pi({i_1})>\pi({i_2})<\pi({i_3})>\cdots \pi({i_k}),$$
where $1\leqslant i_1<i_2<\cdots<i_k\leqslant n$.
Denote by $\as(\pi)$ the length of the longest alternating subsequence of $\pi$.
We also define an {\it up-down run} of $\pi$ to be a run of $\pi$ endowed with a $0$ in the front.
It should be noted that
$\as(\pi)$ equals the number of up-down runs of $\pi$. For example, $$\as(82315467)=\run(082315467)=7.$$
The {\it up-down run polynomials} are defined by
$$S_n(x)=\sum_{\pi\in\msn}x^{\as(\pi)}=\sum_{k=1}^nS(n,k)x^k.$$
Let $S(x;z)=\sum_{n=0}^{\infty}S_n(x)\frac{z^n}{n!}$.
Remarkably, Stanley~\cite[Theorem~2.3]{Sta08} obtained the following closed-form formula:
\begin{equation}\label{Stanley}
S(x;z)=(1-x)\frac{1+\rho+2xe^{\rho z}+(1-\rho)e^{2\rho z}}{1+\rho-x^2+(1-\rho-x^2)e^{2\rho z}},
\end{equation}
where $\rho=\sqrt{1-x^2}$. When $n\geqslant 2$, B\'ona~\cite[Section~1.3.2]{Bona12} obtained a remarkable identity:
\begin{equation}\label{SnxRnx}
S_n(x)=\frac{1}{2}(1+x)R_n(x).
\end{equation}
Ma~\cite[eq.~(14)]{Ma132} found that
\begin{equation*}\label{convolution2}
\sum_{n=0}^\infty \frac{z^n}{n!}\sum_{k=0}^nR(n+1,k)x^{n-k}=\left(\sum_{n=0}^\infty \frac{z^n}{n!}\sum_{k=0}^nS(n,k)x^{n-k}\right)^2.
\end{equation*}

Let $\mbn$ be the hyperoctahedral group of rank $n$. Elements $\pi$ of $\mbn$ are signed permutations of the set $\pm[n]$ such that $\pi(-i)=-\pi(i)$ for all $i$, where $\pm[n]=\{\pm1,\pm2,\ldots,\pm n\}$. In this paper, we always identify a signed permutation $\pi=\pi(1)\cdots\pi(n)$ with the word $\pi(0)\pi(1)\cdots\pi(n)$, where $\pi(0)=0$.
A {\it run} of a signed permutation $\pi$ is defined as a maximal consecutive subword that is monotonic in the order:
$\cdots<\overline{2}<\overline{1}<0<1<2<\cdots$, where we adopt the
standard convention that $\overline{i}$ is being used for $-i$.
The {\it up signed permutations} are defined as signed permutations with $\pi(1)> 0$.
Let $T(n,k)$ denote the number of up signed permutations in $\mbn$ with $k$ alternating runs.
The {\it type $B$ alternating run polynomials} over up signed permutations are defined by
$$T_n(x)=\sum_{k=1}^nT(n,k)x^k.$$
Let $T(x;z)=1+\sum_{n=1}^\infty T_n(x)z^n/n!$.
By the method of characteristic, Chow-Ma~\cite[Theorem~12]{Chow14} found that
\begin{equation}\label{ChowMa}
T(x,z)=\frac{1}{1+x}+\frac{x\sqrt{x-1}}{(1+x)\sqrt{x-\cos(2z\sqrt{x^2-1})-\sqrt{x^2-1}\sin(2z\sqrt{x^2-1})}}.
\end{equation}

The polynomials $R_n(x),S_n(x)$ and $T_n(x)$ satisfy the following recursions
\begin{equation*}
\begin{split}
R_{n+2}(x)&=x(nx+2)R_{n+1}(x)+x\left(1-x^2\right)\frac{\mathrm{d}}{\mathrm{d}x}R_{n+1}(x),\\
S_{n+1}(x)&=x(nx+1)S_{n}(x)+x\left(1-x^2\right)\frac{\mathrm{d}}{\mathrm{d}x}S_{n}(x),\\
T_{n+1}(x)&=(2nx^2+3x-1)T_{n}(x)+2x\left(1-x^2\right)\frac{\mathrm{d}}{\mathrm{d}x}T_{n}(x),
\end{split}
\end{equation*}
with $R_0(x)=R_{1}(x)=S_0(x)=1$ and $T_{1}(x)=x$, see~\cite{Chow14,Hwang20,Zhao,Zhuang16}.
The first few polynomials are given as follows (see~\cite[A059427,~A186370]{Sloane} and~\cite{Ma132,Zhuang16}):
$$R_2(x)=2x,~R_3(x)=2x+4x^2,~R_4(x)=2x+12x^2+10x^3;$$
$$S_1(x)=x,~S_2(x)=x+x^2,~S_3(x)=x+3x^2+2x^3,~S_4(x)=x+7x^2+11x^3+5x^4;$$
$$T_2(x)=x+3x^2,~T_3(x)=x+12x^2+11x^3,~T_4(x)=x+39x^2+95x^3+57x^4.$$

The generating functions~\eqref{CarlitzGF},~\eqref{Stanley} and~\eqref{ChowMa} are all seem complicated.
In the next section, we shall deduce the determinantal representations of $R_n(x),S_n(x)$ and $T_n(x)$.
\section{Up-down run and alternating run polynomials}\label{section3}
\subsection{Up-down run and alternating run polynomials over the symmetric group}
\hspace*{\parindent}

As given by~\eqref{SnxRnx}, the polynomial $R_n(x)$ is closely related to $S_n(x)$, which can also be illustrated by the following grammatical descriptions.
\begin{lemma}[\cite{Ma132}]\label{lemma3}
If $G=\{a\rightarrow ab, b\rightarrow bc, c\rightarrow b^2\}$,
then for $n\geqslant 0$, we have
\begin{equation}\label{derivapoly-4}
D_G^n(a)=ac^nS_n\left(\frac{b}{c}\right),~
D_G^{n}(a^2)=a^2c^nR_{n+1}\left(\frac{b}{c}\right).
\end{equation}
\end{lemma}

\begin{theorem}\label{thmSn}
For any $n\geqslant 1$, let $s_n(x)=(-1)^{n+1}(1+x)(1-x^2)^{\lrf{(n-1)/2}}$ and
let $r_n(x)=(-1)^{n-1}{(x-1)}(1-x^2)^{\lrf{(n-1)/2}}$.
In particular, $s_1(x)=1+x$ and $r_1(x)=x-1$.
Then the polynomial $S_{n+1}(x)$ is expressible as the following
lower Hessenberg determinant of order $n$:
 \begin{equation*}
\begin{vmatrix}\
xs_1(x)&-1&0&\cdots&0&0\\
xs_2(x)&\binom{2}{1}s_1(x)&-1&\cdots&0&0\\
xs_3(x)&\binom{3}{1}s_2(x)&\binom{3}{2}s_1(x)&\cdots&0&0\\
\vdots&\vdots&\vdots&\ddots&\vdots&\vdots\\
xs_{n-1}(x)&\binom{n-1}{1}s_{n-2}(x)&\binom{n-1}{2}s_{n-3}(x)&\cdots&\binom{n-1}{n-2}s_1(x)&-1\\
xs_n(x)&\binom{n}{1}s_{n-1}(x)&\binom{n}{2}s_{n-2}(x)&\cdots&\binom{n}{n-2}s_2(x)&\binom{n}{n-1}s_1(x)
    \end{vmatrix}.
\end{equation*}
While $R_{n+1}(x)$ is expressible as the following
lower Hessenberg determinant of order $n+1$:
 \begin{equation*}
\begin{vmatrix}
1&-1&0&0&\cdots&0&0\\
x-1&s_1(x)&-1&0&\cdots&0&0\\
r_2(x)&s_2(x)&\binom{2}{1}s_1(x)&-1&\cdots&0&0\\
r_3(x)&s_3(x)&\binom{3}{1}s_2(x)&\binom{3}{2}s_1(x)&\cdots&0&0\\
\vdots&\vdots&\vdots&\vdots&\ddots&\vdots&\vdots\\
r_{n-1}(x)&s_{n-1}(x)&\binom{n-1}{1}s_{n-2}(x)&\binom{n-1}{2}s_{n-3}(x)&\cdots&\binom{n-1}{n-2}s_1(x)&-1\\
r_n(x)&s_n(x)&\binom{n}{1}s_{n-1}(x)&\binom{n}{2}s_{n-2}(x)&\cdots&\binom{n}{n-2}s_2(x)&\binom{n}{n-1}s_1(x)
    \end{vmatrix}.
\end{equation*}
\end{theorem}
\begin{proof}
(A) By Lemma~\ref{lemma3}, we see that
$D_G^{n+1}(a)=D_G^{n}(ab)=ac^{n+1}S_{n+1}\left(\frac{b}{c}\right)$.
Since $D_G^{n}(ab)=D_G^{n}\left(\frac{1}{\frac{1}{ab}}\right)$,
let $u(a,b,c)=1$ and $v(a,b,c)=\frac{1}{ab}$. Note that
\begin{equation}\label{DGab}
\begin{aligned}
&D_G(v)=-\frac{b+c}{ab},~D_G\left(\frac{b+c}{ab}\right)=-c\frac{b+c}{ab},\\
&D_G\left(c\frac{b+c}{ab}\right)=-\frac{b+c}{ab}(c^2-b^2),~D_G(c^2-b^2)=0.
\end{aligned}
\end{equation}
For $k\geqslant 1$, $D_G^k(u)=0$, and
\begin{equation}\label{vab}
D_G^{2k}\left(v\right)=c\frac{b+c}{ab}(c^2-b^2)^{k-1},~D_G^{2k+1}\left(v\right)=-\frac{b+c}{ab}(c^2-b^2)^k.
\end{equation}
Letting $a=c=1$ and $b=x$, we get
$$D_G^{n}\left(v\right)|_{a=c=1,b=x}=(-1)^n\frac{1+x}{x}(1-x^2)^{\lrf{(n-1)/2}}.$$
Applying Lemmas~\ref{lemma2}, we arrive at
$$ac^{n+1}S_{n+1}\left(\frac{b}{c}\right)|_{a=c=1,b=x}=S_{n+1}(x)=(-1)^nx^{n+1}{w_n^{(1)}},$$
where the $(n+1)$ order determinant $w_n^{(1)}$ can be given as follows:
 \begin{equation*}
\begin{aligned}
  w_n^{(1)}&=\begin{vmatrix}
1&\frac{1}{x}&0&0&\cdots&0&0\\
0&-\frac{1+x}{x}&\binom{1}{1}\frac{1}{x}&0&\cdots&0&0\\
0&\frac{1+x}{x}&-\binom{2}{1}\frac{1+x}{x}&\binom{2}{2}\frac{1}{x}&\cdots&0&0\\
\vdots&\vdots&\vdots&\vdots&\ddots&\vdots&\vdots\\
0&(-1)^{n-1}\frac{1+x}{x}(1-x^2)^{\lrf{(n-2)/2}}&\cdots&\cdots&\cdots&-\binom{n-1}{n-2}\frac{1+x}{x}&\binom{n-1}{n-1}\frac{1}{x}\\
0&(-1)^n\frac{1+x}{x}(1-x^2)^{\lrf{(n-1)/2}}&\cdots&\cdots&\cdots&\binom{n}{n-2}\frac{1+x}{x}&-\binom{n}{n-1}\frac{1+x}{x}
    \end{vmatrix}\\
&=\begin{vmatrix}\
-\frac{1+x}{x}&\binom{1}{1}\frac{1}{x}&0&\cdots&0&0\\
\frac{1+x}{x}&-\binom{2}{1}\frac{1+x}{x}&\binom{2}{2}\frac{1}{x}&\cdots&0&0\\
\vdots&\vdots&\vdots&\ddots&\vdots&\vdots\\
(-1)^{n-1}\frac{1+x}{x}(1-x^2)^{\lrf{(n-2)/2}}&\cdots&\cdots&\cdots&-\binom{n-1}{n-2}\frac{1+x}{x}&\binom{n-1}{n-1}\frac{1}{x}\\
(-1)^n\frac{1+x}{x}(1-x^2)^{\lrf{(n-1)/2}}&\cdots&\cdots&\cdots&\binom{n}{n-2}\frac{1+x}{x}&-\binom{n}{n-1}\frac{1+x}{x}
    \end{vmatrix}.
\end{aligned}
  \end{equation*}
The last determinant is obtained by expanding $w_n^{(1)}$ along the first column.
Multiplying each column of the last determinant by $-x$ and then multiplying the first column by $x$,
we get the desired expression of $S_{n+1}(x)$.

(B) Note that $$D_G^{n}(a^2)=D_G^{n}\left(\frac{\frac{a}{b}}{\frac{1}{ab}}\right),~D_G\left(\frac{a}{b}\right)=\frac{a}{b}(b-c),~D_G\left(\frac{a}{b}(b-c)\right)=-\frac{ac}{b}(b-c),$$
$$D_G\left(\frac{ac}{b}(b-c)\right)=\frac{a}{b}(b-c)(b^2-c^2)=-\frac{a}{b}(b-c)(c^2-b^2).$$
Let $u(a,b,c)=\frac{a}{b}$ and $v(a,b,c)=\frac{1}{ab}$.
Then $D_G^k(v)$ can be computed as in~\eqref{vab}.
Note that $D_G(b^2-c^2)=0$. For $k\geqslant 1$, we have
 $$D_G^{2k-1}(u)=D_G^{2k-1}\left(\frac{a}{b}\right)=\frac{a}{b}(b-c)(c^2-b^2)^{k-1},D_G^{2k}(u)=D_G^{2k}\left(\frac{a}{b}\right)=-\frac{ac}{b}(b-c)(c^2-b^2)^{k-1}.$$
Letting $a=c=1$ and $b=x$, we get
$$D_G^{n}\left(u\right)|_{a=c=1,b=x}=(-1)^{n-1}\frac{x-1}{x}(1-x^2)^{\lrf{(n-1)/2}}.$$
Applying Lemmas~\ref{lemma2} and~\ref{lemma3}, we arrive at
$$a^2c^nR_{n+1}\left(\frac{b}{c}\right)|_{a=c=1,b=x}=R_{n+1}(x)=(-1)^nx^{n+1}{w_n^{(2)}},$$
where $w_n^{(2)}$ is the following
lower Hessenberg determinant of order $n+1$:
\begin{equation*}
\begin{aligned}
\begin{vmatrix}
\frac{1}{x}&\frac{1}{x}&0&0&\cdots&0&0\\
\frac{x-1}{x}&-\frac{1+x}{x}&\binom{1}{1}\frac{1}{x}&0&\cdots&0&0\\
-\frac{x-1}{x}&\frac{1+x}{x}&-\binom{2}{1}\frac{1+x}{x}&\binom{2}{2}\frac{1}{x}&\cdots&0&0\\
\vdots&\vdots&\vdots&\vdots&\ddots&\vdots&\vdots\\
D_G^{n-1}\left(u\right)|_{a=c=1,b=x}&D_G^{n-1}\left(v\right)|_{a=c=1,b=x}&\cdots&\cdots&\cdots&-\binom{n-1}{n-2}\frac{1+x}{x}&\binom{n-1}{n-1}\frac{1}{x}\\
D_G^{n}\left(u\right)|_{a=c=1,b=x}&D_G^{n}\left(v\right)|_{a=c=1,b=x}&\cdots&\cdots&\cdots&\binom{n}{n-2}\frac{1+x}{x}&-\binom{n}{n-1}\frac{1+x}{x}
    \end{vmatrix}.
\end{aligned}
  \end{equation*}
Multiplying the first column of the above determinant by $x$ and multiplying each of the other columns by $-x$, we get the desired expression of the alternating run polynomial $R_{n+1}(x)$.
\end{proof}

When $n=4$, Theorem~\ref{thmSn} says that
$$S_{5}(x)=
\begin{vmatrix}
x(1+x)&-1&0&0\\
-x(1+x)&\binom{2}{1}(1+x)&-1&0\\
x(1+x)(1-x^2)&-\binom{3}{1}(1+x)&\binom{3}{2}(1+x)&-1\\
-x(1+x)(1-x^2)&\binom{4}{1}(1+x)(1-x^2)&-\binom{4}{2}(1+x)&\binom{4}{3}(1+x)
    \end{vmatrix},$$
$$R_{5}(x)=
\begin{vmatrix}
1&-1&0&0&0\\
x-1&1+x&-1&0&0\\
-(x-1)&-(1+x)&\binom{2}{1}(1+x)&-1&0\\
(x-1)(1-x^2)&(1+x)(1-x^2)&-\binom{3}{1}(1+x)&\binom{3}{2}(1+x)&-1\\
-(x-1)(1-x^2)&-(1+x)(1-x^2)&\binom{4}{1}(1+x)(1-x^2)&-\binom{4}{2}(1+x)&\binom{4}{3}(1+x)
    \end{vmatrix}.$$
Combining~\eqref{Hn-recu} and Theorem~\ref{thmSn}, when
$$\det H_n=S_{n+1}(x),~h_{nn}=n(1+x),~\prod_{j=r}^{n-1}h_{j,j+1}=(-1)^{n-r},~h_{n,r}=\binom{n}{r-1}s_{n-r+1}(x),$$
we obtain
$$S_{n+1}(x)=n(1+x)S_n(x)+\sum_{r=1}^{n-1}\binom{n}{r-1}s_{n-r+1}(x)S_{r}(x).$$
Similarly, combining~\eqref{Hn-recu} and Theorem~\ref{thmSn}, when
$$\det H_{n+1}=R_{n+1}(x),~h_{n+1,n+1}=n(1+x),~h_{n+1,1}=r_n(x),$$
$$\prod_{j=r}^{n}h_{j,j+1}=(-1)^{n-r+1}~\text{for $1\leqslant  r\leqslant n$},~h_{n+1,r}=\binom{n}{r-2}s_{n-r+2}(x)~\text{for $2\leqslant  r\leqslant n$},$$
we arrive at
\begin{align*}
R_{n+1}(x)&=n(1+x)R_n(x)+r_n(x)+\sum_{r=2}^{n}\binom{n}{r-2}s_{n-r+2}R_{r-1}(x)\\
&=r_n(x)+\sum_{r=2}^{n+1}\binom{n}{r-2}s_{n-r+2}R_{r-1}(x).
\end{align*}

After simplifying, we get the following result.
\begin{corollary}\label{cor2}
For any $n\geqslant 1$, we have
\begin{align*}
S_{n+1}(x)&=(1+x)\sum_{r=1}^n\binom{n}{r-1}(-1)^{n-r}(1-x^2)^{\lrf{(n-r)/2}}S_r(x),\\
R_{n+1}(x)&=(-1)^{n-1}{(x-1)}(1-x^2)^{\lrf{(n-1)/2}}+\\&(1+x)\sum_{r=2}^{n+1}\binom{n}{r-2}(-1)^{n-r+1}(1-x^2)^{\lrf{(n-r+1)/2}}R_{r-1}(x),
\end{align*}
which imply that $S_n(x)$ is divisible by $(1+x)^{\lrf{n/2}}$ and $R_n(x)$ is divisible by $(1+x)^{\lrf{n/2}-1}$.
\end{corollary}

We can now conclude the following result.
\begin{theorem}\label{thmsntn2}
Let $f_n(x)=(-1)^{n+1}(1-x^2)^{\lrf{n/2}}$.
The up-down run polynomial $S_{n+1}(x)$
can be expressed as a lower Hessenberg determinant of order $n+1$:
 \begin{equation*}
\begin{vmatrix}
x&-1&0&0&\cdots&0&0\\
xS_1(x)&1&-\binom{1}{1}&0&\cdots&0&0\\
xS_2(x)&-(1-x^2)&\binom{2}{1}&-\binom{2}{2}&\cdots&0&0\\
\vdots&\vdots&\vdots&\vdots&\ddots&\vdots&\vdots\\
xS_{n-1}(x)&{f}_{n-1}(x)&\binom{n-1}{1}{f}_{n-2}(x)&\binom{n-1}{2}{f}_{n-3}(x)&\cdots&\binom{n-1}{n-2}&-\binom{n-1}{n-1}\\
xS_n(x)&{f}_n(x)&\binom{n}{1}{f}_{n-1}(x)&\binom{n}{2}{f}_{n-2}(x)&\cdots&-\binom{n}{n-2}(1-x^2)&\binom{n}{n-1}
    \end{vmatrix}.
  \end{equation*}
While the polynomial $R_{n+1}(x)$ can be expressed as a lower Hessenberg determinant of order $n$:
 \begin{equation*}
\begin{vmatrix}
2x&-1&0&0&\cdots&0&0\\
2xR_2(x)&1&-\binom{1}{1}&0&\cdots&0&0\\
2xR_3(x)&-(1-x^2)&\binom{2}{1}&-\binom{2}{2}&\cdots&0&0\\
\vdots&\vdots&\vdots&\vdots&\ddots&\vdots&\vdots\\
2xR_{n-1}(x)&{f}_{n-2}(x)&\binom{n-2}{1}{f}_{n-3}(x)&\binom{n-2}{2}{f}_{n-4}(x)&\cdots&\binom{n-2}{n-3}&-\binom{n-2}{n-2}\\
2xR_{n}(x)&{f}_{n-1}(x)&\binom{n-1}{1}{f}_{n-2}(x)&\binom{n-1}{2}{f}_{n-3}(x)&\cdots&-\binom{n-1}{n-3}(1-x^2)&\binom{n-1}{n-2}
    \end{vmatrix}.
  \end{equation*}
\end{theorem}
\begin{proof}
(A) Consider the grammar $G=\{a\rightarrow ab, b\rightarrow bc, c\rightarrow b^2\}$. It is clear that
$$D_G^{n+1}(a)=D_G^n(ab)=D_G^n\left(\frac{a}{\frac{1}{b}}\right).$$
Let $u(a,b,c)=a$ and $v(a,b,c)=\frac{1}{b}$. From Lemma~\ref{lemma3}, we see that
$D_G^n(u)|_{a=c=1,b=x}=S_n\left(x\right)$.
Note that
$$D_G(v)=D_G\left(\frac{1}{b}\right)=-\frac{c}{b},~D_G\left(\frac{c}{b}\right)=\frac{b^2-c^2}{b}=-\frac{c^2-b^2}{b}.$$
Then for $k\geqslant 1$, the higher derivatives of $v$ are given as follows:
$$D_G^{2k-1}(v)=-\frac{c}{b}(c^2-b^2)^{k-1},~D_G^{2k}(v)=\frac{1}{b}(c^2-b^2)^{k}.$$
Setting $a=c=1$ and $b=x$, we obtain
\begin{equation}\label{v}
D_G^{2k-1}(v)|_{a=c=1,b=x}=-\frac{1}{x}(1-x^2)^{k-1},~D_G^{2k}(v)|_{a=c=1,b=x}=\frac{1}{x}(1-x^2)^{k}.
\end{equation}
We define $\widetilde{f}_n(x)=D_G^{n}(v)|_{a=c=1,b=x}$
Then $$\widetilde{f}_n(x)=(-1)^n\frac{1}{x}(1-x^2)^{\lrf{n/2}}.$$
It follows from Lemma~\ref{lemma2} that $S_{n+1}(x)=(-1)^nx^{n+1}w_n^{(1)}$, where $w_n^{(1)}$ is the following
lower Hessenberg determinant of order $n+1$:
 \begin{equation*}
  w_n^{(1)}=\begin{vmatrix}
1&\frac{1}{x}&0&0&\cdots&0&0\\
S_1(x)&-\frac{1}{x}&\binom{1}{1}\frac{1}{x}&0&\cdots&0&0\\
S_2(x)&\frac{1}{x}(1-x^2)&-\binom{2}{1}\frac{1}{x}&\binom{2}{2}\frac{1}{x}&\cdots&0&0\\
\vdots&\vdots&\vdots&\vdots&\ddots&\vdots&\vdots\\
S_{n-1}(x)&\widetilde{f}_{n-1}(x)&\binom{n-1}{1}\widetilde{f}_{n-2}(x)&\binom{n-1}{2}\widetilde{f}_{n-3}(x)&\cdots&-\binom{n-1}{n-2}\frac{1}{x}&\binom{n-1}{n-1}\frac{1}{x}\\
S_n(x)&\widetilde{f}_n(x)&\binom{n}{1}\widetilde{f}_{n-1}(x)&\binom{n}{2}\widetilde{f}_{n-2}(x)&\cdots&\binom{n}{n-2}\frac{1}{x}(1-x^2)&-\binom{n}{n-1}\frac{1}{x}
    \end{vmatrix}.
  \end{equation*}
First multiplying each column by $x$, and then except the first column, multiplying each of the other columns by $-1$,
we see that $(-1)^nx^{n+1}w_n^{(1)}$ can be expressed as follows:
 \begin{equation*}
\begin{vmatrix}
x&-1&0&0&\cdots&0&0\\
xS_1(x)&1&-\binom{1}{1}&0&\cdots&0&0\\
xS_2(x)&-(1-x^2)&\binom{2}{1}&-\binom{2}{2}&\cdots&0&0\\
\vdots&\vdots&\vdots&\vdots&\ddots&\vdots&\vdots\\
xS_{n-1}(x)&{f}_{n-1}(x)&\binom{n-1}{1}{f}_{n-2}(x)&\binom{n-1}{2}{f}_{n-3}(x)&\cdots&\binom{n-1}{n-2}&-\binom{n-1}{n-1}\\
xS_n(x)&{f}_n(x)&\binom{n}{1}{f}_{n-1}(x)&\binom{n}{2}{f}_{n-2}(x)&\cdots&-\binom{n}{n-2}(1-x^2)&\binom{n}{n-1}
    \end{vmatrix},
  \end{equation*}
where $f_n(x)=(-x)\widetilde{f}_n(x)=(-1)^{n+1}(1-x^2)^{\lrf{n/2}}$.

(B) For $n\geqslant 1$, it is clear that $$D_G^{n}(a^2)=D_G^{n-1}(2a^2b)=2D_G^{n-1}(a^2b)=2D_G^{n-1}\left(\frac{a^2}{\frac{1}{b}}\right).$$
Let $u(a,b,c)=a^2$ and $v(a,b,c)=\frac{1}{b}$. Then for $k\geqslant 1$, it follows from~\eqref{derivapoly-4} and~\eqref{v} that
$$D_G^{k}(u)|_{a=c=1,b=x}=R_{k+1}\left(x\right),~D_G^{k}(v)|_{a=c=1,b=x}=(-1)^k\frac{1}{x}(1-x^2)^{\lrf{k/2}}.$$
Applying Lemma~\ref{lemma2}, we see that $R_{n+1}(x)=2(-1)^{n-1}x^{n}w_n^{(2)}$, where $w_n^{(2)}$ is a
lower Hessenberg determinant of order $n$. Similarly, multiplying the first column by $2x$ and multiplying each of the other columns by $-x$,
we immediately get the desired representation of $R_{n+}(x)$.
\end{proof}

Combining~\eqref{Hn-recu} and Theorem~\ref{thmsntn2}, it is routine to verify the following result.
\begin{corollary}
For $n\geqslant 1$, we have
\begin{align*}
S_{n+1}(x)&=(n+x)S_n(x)+(1-x^2)\sum_{r=2}^n\binom{n}{r-2}(-1)^{n-r+1}(1-x^2)^{\lrf{(n-r)/2}}S_{r-1}(x),\\
R_{n+1}(x)&=2xR_n(x)+\sum_{r=2}^{n}\binom{n-1}{r-2}(-1)^{n-r}(1-x^2)^{\lrf{(n-r+1)/2}}R_r(x),
\end{align*}
which imply that $S_n(x)$ is divisible by $(1+x)^{\lrf{n/2}}$ and $R_n(x)$ is divisible by $(1+x)^{\lrf{n/2}-1}$.
\end{corollary}
\subsection{On the type $B$ alternating run polynomials over up signed permutations}
\hspace*{\parindent}

Let $T_n(x)=\sum_{k=1}^nT(n,k)x^k$. For $n\geqslant 2$ and $1\leqslant k\leqslant n$,
Zhao~\cite[Theorem~4.2.1]{Zhao} showed that the numbers $T(n,k)$ satisfy the
the following recurrence relation
\begin{equation}\label{tnk-recurrence}
T(n,k)=(2k-1)T(n-1,k)+3T(n-1,k-1)+(2n-2k+2)T(n-1,k-2),
\end{equation}
where $T(1,1)=1$ and $T(1,k)=0$ for $k>1$.
We now present a grammatical description.
\begin{lemma}\label{lemma5}
If $G=\{a\rightarrow a(b+c), b\rightarrow 2bc, c\rightarrow 2b^2\}$,
then for $n\geqslant 1$, we have
\begin{equation}\label{derivapoly-6}
D_G^n(a)|_{a=b=x,c=1}=(1+x)T_n(x).
\end{equation}
\end{lemma}
\begin{proof}
Note that $D_G(a)=a(b+c),~D_G^2(a)=a(b+c)(c+3b)$ and $D_G^3(a)=a(b+c)(c^2+12bc+11b^2)$.
So the result holds for $n\leqslant 3$. We proceed by induction.
Assume that
\begin{equation}\label{DGaT}
D_G^n(a)=a(b+c)\sum_{k=1}^nT(n,k)b^{k-1}c^{n-k}.
\end{equation}
Then
\begin{align*}
D_G^{n+1}(a)&=D\left(a(b+c)\sum_{k=1}^nT(n,k)b^{k-1}c^{n-k}\right)\\
&=a(b+c)(c+3b)\sum_{k}T(n,k)b^{k-1}c^{n-k}+\\
&a(b+c)\sum_{k=1}^nT(n,k)\left[2(k-1)b^{k-1}c^{n-k+1}+2(n-k)b^{k+1}c^{n-k-1}\right].
\end{align*}
Exacting the coefficients of $a(b+c)b^{k-1}c^{n-k+1}$ in the last expression, we obtain
$$T(n,k)+3T(n,k-1)+2(k-1)T(n,k)+2(n-k+2)T(n,k-2)=T(n+1,k).$$
It follows from~\eqref{tnk-recurrence} that
$D_G^{n+1}(a)=a(b+c)\sum_{k=1}^{n+1}T(n+1,k)b^{k-1}c^{n-k+1}$, as desired.
\end{proof}
\begin{theorem}\label{thmT}
For any $n\geqslant 1$, let $${t}_n(x)=\left\{
                       \begin{array}{ll}
                         (1+x)(1-x^2)^{k-1}, & \hbox{if $n=2k-1$;} \\
                         -(1-x^2)^k, & \hbox{if $n=2k$.}
                       \end{array}
                     \right.
$$
Then $T_{n}(x)$ is expressible as the following
lower Hessenberg determinant of order $n$:
\begin{equation*}
\begin{vmatrix}
x&-1&0&\cdots&0&0\\
-x(1-x)&\binom{2}{1}(1+x)&-1&\cdots&0&0\\
\frac{x}{1+x}t_3(x)&\binom{3}{1}t_2(x)&\binom{3}{2}(1+x)&\cdots&0&0\\
\vdots&\vdots&\vdots&\ddots&\vdots&\vdots\\
\frac{x}{1+x}{t}_{n-1}(x)&\binom{n-1}{1}{t}_{n-2}(x)&\binom{n-1}{2}{t}_{n-3}(x)&\cdots&\binom{n-1}{n-2}(1+x)&-1\\
\frac{x}{1+x}{t}_n(x)&\binom{n}{1}{t}_{n-1}(x)&\binom{n}{2}{t}_{n-2}(x)&\cdots&-\binom{n}{n-2}(1-x^2)&\binom{n}{n-1}(1+x)
    \end{vmatrix}.
  \end{equation*}
\end{theorem}
\begin{proof}
Consider $G=\{a\rightarrow a(b+c), b\rightarrow 2bc, c\rightarrow 2b^2\}$.
By Lemma~\ref{lemma5}, we find that
$$D_G^n(a)=D_G^n\left(\frac{1}{\frac{1}{a}}\right).$$
Let $u(a,b,c)=1$ and $v(a,b,c)=\frac{1}{a}$. Then $D_G^k(u)=0$ for $k\geqslant 1$.
Note that
$$D_G(v)=D_G\left(\frac{1}{a}\right)=-\frac{b+c}{a},~D_G\left(\frac{b+c}{a}\right)=\frac{b^2-c^2}{a}=-\frac{c^2-b^2}{a}.$$
Since $D_G(c^2-b^2)=0$, it follows that for $k\geqslant 1$, we have
\begin{equation}\label{DG2kv}
D_G^{2k-1}(v)=-\frac{b+c}{a}(c^2-b^2)^{k-1},~D_G^{2k}(v)=\frac{1}{a}(c^2-b^2)^k.
\end{equation}

When $a=b=x$ and $c=1$,
$$D_G^{2k-1}(v)|_{a=b=x,c=1}=-\frac{1+x}{x}(1-x^2)^{k-1},~D_G^{2k}(v)|_{a=b=x,c=1}=\frac{1}{x}(1-x^2)^k.$$
For $n\geqslant 1$, we define
$$\widetilde{t}_n(x)=\left\{
                       \begin{array}{ll}
                         -\frac{1+x}{x}(1-x^2)^{k-1}, & \hbox{if $n=2k-1$;} \\
                         \frac{1}{x}(1-x^2)^k, & \hbox{if $n=2k$.}
                       \end{array}
                     \right.
$$
Applying Lemma~\ref{lemma2} and~\eqref{DGaT}, we arrive at
$$a(b+c)\sum_{k=1}^nT(n,k)b^{k-1}c^{n-k}|_{a=b=x,c=1}=(1+x)T_{n}(x)=(-1)^nx^{n+1}{w_n},$$
where the $(n+1)$ order lower Hessenberg determinant $w_n$ is given by
\begin{equation*}
\begin{vmatrix}
1&\frac{1}{x}&0&0&\cdots&0&0\\
0&-\frac{1+x}{x}&\binom{1}{1}\frac{1}{x}&0&\cdots&0&0\\
0&\frac{1}{x}(1-x^2)&-\binom{2}{1}\frac{1+x}{x}&\binom{2}{2}\frac{1}{x}&\cdots&0&0\\
\vdots&\vdots&\vdots&\vdots&\ddots&\vdots&\vdots\\
0&\widetilde{t}_{n-1}(x)&\binom{n-1}{1}\widetilde{t}_{n-2}(x)&\binom{n-1}{2}\widetilde{t}_{n-3}(x)&\cdots&-\binom{n-1}{n-2}\frac{1+x}{x}&\binom{n-1}{n-1}\frac{1}{x}\\
0&\widetilde{t}_n(x)&\binom{n}{1}\widetilde{t}_{n-1}(x)&\binom{n}{2}\widetilde{t}_{n-2}(x)&\cdots&\binom{n}{n-2}\frac{1}{x}(1-x^2)&-\binom{n}{n-1}\frac{1+x}{x}
    \end{vmatrix}.
  \end{equation*}
Expanding along the first column of $w_n$, we see that
\begin{equation*}
w_n=\begin{vmatrix}
-\frac{1+x}{x}&\binom{1}{1}\frac{1}{x}&0&\cdots&0&0\\
\frac{1}{x}(1-x^2)&-\binom{2}{1}\frac{1+x}{x}&\binom{2}{2}\frac{1}{x}&\cdots&0&0\\
\vdots&\vdots&\vdots&\ddots&\vdots&\vdots\\
\widetilde{t}_{n-1}(x)&\binom{n-1}{1}\widetilde{t}_{n-2}(x)&\binom{n-1}{2}\widetilde{t}_{n-3}(x)&\cdots&-\binom{n-1}{n-2}\frac{1+x}{x}&\binom{n-1}{n-1}\frac{1}{x}\\
\widetilde{t}_n(x)&\binom{n}{1}\widetilde{t}_{n-1}(x)&\binom{n}{2}\widetilde{t}_{n-2}(x)&\cdots&\binom{n}{n-2}\frac{1}{x}(1-x^2)&-\binom{n}{n-1}\frac{1+x}{x}
    \end{vmatrix}.
  \end{equation*}
Multiplying each column of the above determint by $-x$, we see that
\begin{equation*}
(-1)^nx^nw_n=\begin{vmatrix}
1+x&-1&0&\cdots&0&0\\
-(1-x^2)&\binom{2}{1}(1+x)&-1&\cdots&0&0\\
\vdots&\vdots&\vdots&\ddots&\vdots&\vdots\\
{t}_{n-1}(x)&\binom{n-1}{1}{t}_{n-2}(x)&\binom{n-1}{2}{t}_{n-3}(x)&\cdots&\binom{n-1}{n-2}(1+x)&-1\\
{t}_n(x)&\binom{n}{1}{t}_{n-1}(x)&\binom{n}{2}{t}_{n-2}(x)&\cdots&-\binom{n}{n-2}(1-x^2)&\binom{n}{n-1}(1+x)
    \end{vmatrix},
  \end{equation*}
where
$${t}_n(x)=\left\{
                       \begin{array}{ll}
                         (1+x)(1-x^2)^{k-1}, & \hbox{if $n=2k-1$;} \\
                         -(1-x^2)^k, & \hbox{if $n=2k$.}
                       \end{array}
                     \right.
$$
Multiplying the first column by $\frac{x}{1+x}$ leads to the stated formula.
\end{proof}

When $n=4$, Theorem~\ref{thmT} says that
$$\frac{1}{x}T_4(x)=\begin{vmatrix}
1&-1&0&0\\
-(1-x)&\binom{2}{1}(1+x)&-1&0\\
1-x^2&-\binom{3}{1}(1-x^2)&\binom{3}{2}(1+x)&-1\\
-(1-x)(1-x^2)&\binom{4}{1}(1+x)(1-x^2)&-\binom{4}{2}(1-x^2)&\binom{4}{3}(1+x)
    \end{vmatrix}.$$
Combining~\eqref{Hn-recu} and Theorem~\ref{thmT}, we get the following.
\begin{corollary}
For any $n\geqslant 1$, let $${t}_n(x)=\left\{
                       \begin{array}{ll}
                         (1+x)(1-x^2)^{k-1}, & \hbox{if $n=2k-1$;} \\
                         -(1-x^2)^k, & \hbox{if $n=2k$.}
                       \end{array}
                     \right.
$$
Then we have
$$T_n(x)=\frac{x}{1+x}t_n(x)+\sum_{r=2}^n\binom{n}{r-1}t_{n-r+1}T_{r-1}(x),$$
which implies that the multiplicity of $-1$ in $T_n(x)$ is $\lrf{\frac{n-1}{2}}$.
\end{corollary}

We now give a deep connection between the up-down run polynomials $S_n(x)$ and the type $B$ alternating run polynomials $T_n(x)$.
\begin{theorem}\label{TS}
For any $n\geqslant 1$, let $${t}_n(x)=\left\{
                       \begin{array}{ll}
                         (1+x)(1-x^2)^{k-1}, & \hbox{if $n=2k-1$;} \\
                         -(1-x^2)^k, & \hbox{if $n=2k$.}
                       \end{array}
                     \right.
$$
Then $T_{n+1}(x)$ is expressible as the following
lower Hessenberg determinant of order $n+1$:
\begin{equation*}
\begin{vmatrix}
x&-1&0&0&\cdots&0&0\\
2xS_1(x)&{1+x}&-\binom{1}{1}&0&\cdots&0&0\\
2^2xS_2(x)&-(1-x^2)&\binom{2}{1}(1+x)&-\binom{2}{2}&\cdots&0&0\\
\vdots&\vdots&\vdots&\vdots&\ddots&\vdots&\vdots\\
2^{n-1}xS_{n-1}(x)&{t}_{n-1}(x)&\binom{n-1}{1}{t}_{n-2}(x)&\binom{n-1}{2}{t}_{n-3}(x)&\cdots&\binom{n-1}{n-2}(1+x)&-\binom{n-1}{n-1}\\
2^{n}xS_n(x)&{t}_n(x)&\binom{n}{1}{t}_{n-1}(x)&\binom{n}{2}{t}_{n-2}(x)&\cdots&-\binom{n}{n-2}(1-x^2)&\binom{n}{n-1}(1+x)
    \end{vmatrix}.
  \end{equation*}
\end{theorem}
\begin{proof}
Recall Lemma~\ref{lemma5}.
It follows from~\eqref{DGaT} that
\begin{equation}\label{DGaT02}
D_G^{n+1}(a)=a(b+c)\sum_{k=1}^{n+1}T(n+1,k)b^{k-1}c^{n+1-k}.
\end{equation}
On the other hand,
$$D_G^{n+1}(a)=D_G^n(a(b+c))=D_G^n\left(\frac{b+c}{{1}/{a}}\right).$$
Let $u(a,b,c)=b+c$ and $v(a,b,c)=\frac{1}{a}$. For any $k\geqslant 1$,
the expansion of $D_G^k(v)$ is given by~\eqref{DG2kv}. When $u=b+c$, then $D_G(u)=2b(b+c)$ and for $k\geqslant 2$,
we observe that
\begin{equation}\label{DNGu}
D_G^k(u)=2^{k-1}(b+c)^2c^{k-1}R_k\left(\frac{b}{c}\right)=2^{k-1}(b+c)^2\sum_{i=1}^{k-1}R(k,i)b^ic^{k-1-i},
\end{equation}
where $R(k,i)$ is defined by~\eqref{Rnk-recurrence}, i.e., the number of permutations in $\ms_k$ with $i$ alternating runs.

We now prove~\eqref{DNGu} by induction. Clearly, $D_G^2(u)=2(b+c)^22b$ and $D_G^3(u)=4(b+c)^2(2bc+4b^2)$ and
$D_G^4(u)=8(b+c)^2(2bc^2+12b^2c+10b^3)$. Thus~\eqref{DNGu} holds for $k\leqslant 4$. Assume it holds for $k=n$. Then
\begin{align*}
D_G^{n+1}(u)&=D_G\left(2^{n-1}(b+c)^2\sum_{k=1}^nR(n,k)b^kc^{n-1-k}\right)\\
&=2^{n-1}D_G\left((b+c)^2\sum_{k=1}^nR(n,k)b^kc^{n-1-k}\right)\\
&=2^{n-1}\left(4b(b+c)^2\sum_{k=1}^nR(n,k)b^kc^{n-1-k}\right)+\\
&2^{n-1}(b+c)^2\sum_{k}R(n,k)\left(2kb^{k}c^{n-k}+2(n-1-k)b^{k+2}c^{n-2-k}\right)
\end{align*}
Exacting the coefficient of $2^{n}(b+c)^2b^kc^{n-k}$, we get
$$2R(n,k-1)+kR(n,k)+(n-k+1)R(n,k-2)=R(n+1,k),$$
as desired. Thus~\eqref{DNGu} holds for $k=n+1$.
When $b=x$ and $c=1$, then $u|_{b=x,c=1}=1+x$, $D_G(u)|_{b=x,c=1}=2x(1+x)$ and~\eqref{DNGu} reduces to
\begin{equation*}
D_G^n(u)|_{b=x,c=1}=2^{n-1}(1+x)^2R_n(x).
\end{equation*}

When $a=b=x$ and $c=1$, applying Lemma~\ref{lemma2} and~\eqref{DGaT02}, we see that
$(1+x)T_{n+1}(x)=(-1)^nx^{n+1}w_n$, where $w_n$ is given by
\begin{equation*}
\begin{vmatrix}
1+x&\frac{1}{x}&0&0&\cdots&0&0\\
2x(1+x)&-\frac{1+x}{x}&\binom{1}{1}\frac{1}{x}&0&\cdots&0&0\\
2(1+x)^2R_2(x)&\frac{1}{x}(1-x^2)&-\binom{2}{1}\frac{1+x}{x}&\binom{2}{2}\frac{1}{x}&\cdots&0&0\\
\vdots&\vdots&\vdots&\vdots&\ddots&\vdots&\vdots\\
2^{n-2}(1+x)^2R_{n-1}(x)&\widetilde{t}_{n-1}(x)&\binom{n-1}{1}\widetilde{t}_{n-2}(x)&\binom{n-1}{2}\widetilde{t}_{n-3}(x)&\cdots&-\binom{n-1}{n-2}\frac{1+x}{x}&\binom{n-1}{n-1}\frac{1}{x}\\
2^{n-1}(1+x)^2R_n(x)&\widetilde{t}_n(x)&\binom{n}{1}\widetilde{t}_{n-1}(x)&\binom{n}{2}\widetilde{t}_{n-2}(x)&\cdots&\binom{n}{n-2}\frac{1}{x}(1-x^2)&-\binom{n}{n-1}\frac{1+x}{x}
    \end{vmatrix},
  \end{equation*}
where
$$\widetilde{t}_n(x)=\left\{
                       \begin{array}{ll}
                         -\frac{1+x}{x}(1-x^2)^{k-1}, & \hbox{if $n=2k-1$;} \\
                         \frac{1}{x}(1-x^2)^k, & \hbox{if $n=2k$.}
                       \end{array}
                     \right.
$$
Multiplying the first column of $w_n$ by $\frac{x}{1+x}$, multiplying each of the other columns by $-x$, and applying~\eqref{SnxRnx},
we get the desired result. This completes the proof.
\end{proof}

By~\eqref{Hn-recu} and Theorem~\ref{TS}, we find the following result. The proof is omitted for simplicity.
\begin{corollary}
For $n\geqslant 1$, let
$${t}_n(x)=\left\{
                       \begin{array}{ll}
                         (1+x)(1-x^2)^{k-1}, & \hbox{if $n=2k-1$;} \\
                         -(1-x^2)^k, & \hbox{if $n=2k$.}
                       \end{array}
                     \right.
$$
Then
the type $B$ alternating run polynomials satisfy the following recursion:
\begin{equation}\label{TSR}
\begin{aligned}
T_{n+1}(x)&=2^nxS_n(x)+\sum_{r=2}^{n+1}\binom{n}{r-2}t_{n-r+2}T_{r-1}(x)\\
&=2^{n-1}x(1+x)R_n(x)+\sum_{r=2}^{n+1}\binom{n}{r-2}t_{n-r+2}T_{r-1}(x),
\end{aligned}
\end{equation}
where $S_n(x)$ are the up-down polynomials and $R_n(x)$ are the type $A$ alternating run polynomials.
\end{corollary}
\section{On the alternating run polynomials over dual Stirling permutations}\label{dualStirling}
For a signed permutation $\sigma\in\mbn$, let
$$\operatorname{des}_B(\sigma)=\#\{i\in\{0,1,2,\ldots,n-1\}:~\sigma(i)>\sigma({i+1}),~\sigma(0)=0\}$$
be the number of {\it descents} of $\sigma$.
The {\it type $B$ Eulerian polynomials} are defined by
$$B_n(x)=\sum_{\sigma\in \mbn}x^{\operatorname{des}_B(\sigma)}.$$
They satisfy the recursion
$$B_{n+1}(x)=(2nx+1+x)B_{n}(x)+2x(1-x)\frac{\mathrm{d}}{\mathrm{d}x}B_{n}(x),$$
with $B_0(x)=1,~B_1(x)=1+x$ and $B_2(x)=1+6x+x^2$.
Let $\Stirling{n}{k}$ be the Stirling number of the second kind, i.e., the number of ways to partition $[n]$ into $k$ blocks.
Brenti~\cite[Theorem~3.14]{Bre94} found an explicit formula:
\begin{equation*}\label{Frobenius02}
B_n(x)=\sum_{k=0}^nk!\sum_{i=k}^n\binom{n}{i}2^i\Stirling{i}{k}(x-1)^{n-k}.
\end{equation*}

Let $r(x)=\sqrt{\frac{1+x}{1-x}}$. Following~\cite[p.~14]{Ma20},
the {\it alternating run polynomials of dual Stirling permutations} can be defined by
\begin{align*}
&\left(x\frac{\mathrm{d}}{\mathrm{d}x}\right)^{n}r(x)=\frac{r(x)F_{n}(x)}{(1-x^2)^{n}}.
\end{align*}
They satisfy the recursion
\begin{equation*}
F_{n+1}(x)=(x+2nx^2)F_n(x)+x(1-x^2)\frac{\mathrm{d}}{\mathrm{d}x}F_n(x),
\end{equation*}
with $F_0(x)=1$
The first few alternating run polynomials of dual Stirling permutations are
\begin{align*}
  F_1(x)&=x, \\
  F_2(x)&=x+x^2+x^3, \\
  F_3(x)&=x+3x^2+7x^3+3x^4+x^5,\\
  F_4(x)&=x+7x^2+29x^3+31x^4+29x^5+7x^6+x^7.
\end{align*}

Stirling permutations were introduced by Gessel and Stanley~\cite{Gessel78} and have been extensively studied in recent years~\cite{Ma23}.
A {\it Stirling permutation} of order $n$ is a permutation of the multiset $\{1,1,2,2,\ldots,n,n\}$ such that
for each $i$, $1\leq i\leq n$, all entries between the two occurrences of $i$ are larger than $i$.
Denote by $\mqn$ the set of {\it Stirling permutations} of order $n$.
Let $\sigma=\sigma_1\sigma_2\cdots\sigma_{2n}\in\mqn$.
Let $\Phi$ be the injection which maps each first occurrence of entry $j$ in $\sigma$ to $2j$ and the
second occurrence of $j$ to $2j-1$,
where $j\in [n]$. For example, $\Phi(221331)=432651$.
Let $\Phi(\mqn)=\{\pi\mid \sigma\in\mqn, \Phi(\sigma)=\pi\}$
be the set of {\it dual Stirling permutations} of order $n$.
A main result in~\cite{Mawang16} says that
$$F_n(x)=\sum_{\pi\in\Phi(\mqn)}x^{\run(\pi)}.$$
Let $F(x;z)=\sum_{n=0}^\infty F_n(x)\frac{z^n}{n!}$. According to~\cite[Eq.~(3.7)]{Mawang16}, we have
\begin{equation*}\label{Txz-EGF}
F(x;z)=\frac{{e^{z \left( x-1 \right)  \left( x+1 \right) }}+x}{1+x}\sqrt {{\frac {1-x^2}{{e^{2\,z \left( x-1
\right) \left( x+1 \right) }}-x^2}}}.
\end{equation*}

An occurrence of an {\it ascent-plateau} of $\sigma\in\mqn$ is an index $i$ such that $\sigma_{i-1}<\sigma_{i}=\sigma_{i+1}$, where $i\in\{2,3,\ldots,2n-1\}$.
Let $\ap(\sigma)$ be the number of ascent-plateaus of $\sigma$.
The number of {\it flag ascent-plateaus} of $\sigma$ is defined by $$\fap(\sigma)=\left\{
               \begin{array}{ll}
                 2\ap(\sigma)+1, & \hbox{if $\sigma_1=\sigma_2$;} \\
                 2\ap(\sigma), & \hbox{otherwise.}
               \end{array}
             \right.
$$
It is easy to check that $\fap(\sigma)=\altrun( \Phi(\sigma))$ for any $\sigma\in\mqn$.
Hence another interpretation of $F_n(x)$ is given as follows: $$F_n(x)=\sum_{\sigma\in\mqn}x^{\fap(\sigma)}.$$

A unified grammatical interpretations of $B_n(x)$ and $F_n(x)$ is given as follows.
\begin{lemma}[\cite{Ma131,Mawang16,Ma20}]\label{Gabc}
If $G=\{a\rightarrow abc, b\rightarrow bc^2, c\rightarrow b^2c\}$,
then we have
\begin{equation*}\label{DnxFnk}
D_G^n(bc)=bc^{2n+1}B_n\left(\frac{b^2}{c^2}\right),~D_{G}^n(a)=ac^{2n}F_n\left(\frac{b}{c}\right).
\end{equation*}
\end{lemma}
\begin{theorem}\label{BT}
For any $n\geqslant 1$, the polynomial $F_{n+1}(x)$ can be expressed as the
following lower Hessenberg determinant of order $n+1$:
\begin{equation*}
\begin{small}
\begin{vmatrix}
1&-1&0&0&\cdots&0&0\\
xB_1(x^2)&-F_1(-x)&-1&0&\cdots&0&0\\
xB_2(x^2)&-F_2(-x)&-\binom{2}{1}F_1(-x)&-1&\cdots&0&0\\
\vdots&\vdots&\vdots&\vdots&\ddots&\vdots&\vdots\\
xB_{n-1}(x^2)&-F_{n-1}(-x)&-\binom{n-1}{1}F_{n-2}(-x)&-\binom{n-1}{2}F_{n-3}(-x)&\cdots&-\binom{n-1}{n-2}F_1(-x)&-1\\
xB_n(x^2)&-F_n(-x)&-\binom{n}{1}F_{n-1}(-x)&-\binom{n}{2}F_{n-2}(-x)&\cdots&-\binom{n}{n-2}F_2(-x)&-\binom{n}{n-1}F_1(-x)
    \end{vmatrix}.
\end{small}
  \end{equation*}
\end{theorem}
\begin{proof}
Consider the grammar $G=\{a\rightarrow abc, b\rightarrow bc^2, c\rightarrow b^2c\}$.
Note that
\begin{equation}\label{abc}
ac^{2n+2}F_{n+1}\left(\frac{b}{c}\right)=D_G^{n+1}(a)=D_G^{n}(abc)=D_G^n\left(\frac{bc}{\frac{1}{a}}\right).
\end{equation}
Let $u(a,b,c)=bc$ and $v(a,b,c)=\frac{1}{a}$. By Lemma~\ref{Gabc}, we have
$D_G^n(u)=bc^{2n+1}B_n\left(\frac{b^2}{c^2}\right)$.
Note that
$$D_G(v)=D_G\left(\frac{1}{a}\right)=\frac{c}{a}(-b),~D_G^2(v)=\frac{c}{a}(-bc^2+b^2c-b^3),$$
$$D_G^3(v)=\frac{c}{a}(-bc^4+3b^2c^3-7b^3c^2+3b^4c-b^5).$$
By induction, it is routine to verify that
$$D_G^n(v)=\frac{c^{2n}}{a}F_n\left(-\frac{b}{c}\right),$$
and we omit the details for simplicity. In the case when $a=c=1$ and $b=x$,
applying~\eqref{abc} and Lemma~\ref{lemma2}, we arrive at $F_{n+1}(x)=(-1)^n{w_n}$,
where $w_n$ is given by
 \begin{equation*}
\begin{vmatrix}
1&1&0&0&\cdots&0&0\\
xB_1(x^2)&F_1(-x)&\binom{1}{1}&0&\cdots&0&0\\
xB_2(x^2)&F_2(-x)&\binom{2}{1}F_1(-x)&\binom{2}{2}&\cdots&0&0\\
\vdots&\vdots&\vdots&\vdots&\ddots&\vdots&\vdots\\
xB_{n-1}(x^2)&F_{n-1}(-x)&\binom{n-1}{1}F_{n-2}(-x)&\binom{n-1}{2}F_{n-3}(-x)&\cdots&\binom{n-1}{n-2}F_1(-x)&\binom{n-1}{n-1}\\
xB_n(x^2)&F_n(-x)&\binom{n}{1}F_{n-1}(-x)&\binom{n}{2}F_{n-2}(-x)&\cdots&\binom{n}{n-2}F_2(-x)&\binom{n}{n-1}F_1(-x)
    \end{vmatrix}.
  \end{equation*}
Except for the first column, multiplying each of the other columns by $-1$ gives the desired result. This completes the proof.
\end{proof}

By~\eqref{Hn-recu} and Theorem~\ref{BT}, we immediately find the following result.
\begin{corollary}
For $n\geqslant 0$, we have
$$F_{n+1}(x)=xB_n(x^2)+nxF_n(x)-\sum_{r=2}^n\binom{n}{r-2}F_{n-r+2}(-x)F_{r-1}(x).$$
\end{corollary}

A dual of Theorem~\ref{BT} is given as follows.
\begin{theorem}\label{Ff02}
For any $n\geqslant 0$, let
$${f}_n(x)=\left\{
                             \begin{array}{ll}
                              {(1+x^2)}(x^2-1)^{2k-2}, & \hbox{if $n=2k-1$;} \\
                              -(x^2-1)^{2k}, & \hbox{if $n=2k$.}
                             \end{array}
                           \right.
$$
Then $F_{n+1}(x)$ can be expressed as the
following lower Hessenberg determinant of order $n+1$:
 \begin{equation*}
\begin{vmatrix}
x&-1&0&0&\cdots&0&0\\
xF_1(x)&{f}_1(x)&-\binom{1}{1}&0&\cdots&0&0\\
xF_2(x)&{f}_2(x)&\binom{2}{1}{f}_1(x)&-\binom{2}{2}&\cdots&0&0\\
\vdots&\vdots&\vdots&\vdots&\ddots&\vdots&\vdots\\
xF_{n-1}(x)&{f}_{n-1}(x)&\binom{n-1}{1}{f}_{n-2}(x)&\binom{n-1}{2}{f}_{n-3}(x)&\cdots&\binom{n-1}{n-2}{f}_1(x)&-\binom{n-1}{n-1}\\
xF_n(x)&{f}_n(x)&\binom{n}{1}{f}_{n-1}(x)&\binom{n}{2}{f}_{n-2}(x)&\cdots&\binom{n}{n-2}{f}_2(x)&\binom{n}{n-1}{f}_1(x)
    \end{vmatrix}.
  \end{equation*}

\end{theorem}\label{F02}
\begin{proof}
Consider the grammar $G=\{a\rightarrow abc, b\rightarrow bc^2, c\rightarrow b^2c\}$.
Note that
\begin{equation}\label{abc02}
ac^{2n+2}F_{n+1}\left(\frac{b}{c}\right)=D_G^{n+1}(a)=D_G^{n}(abc)=D_G^n\left(\frac{a}{\frac{1}{bc}}\right).
\end{equation}
Let $u(a,b,c)=a$ and $v(a,b,c)=\frac{1}{bc}$. By Lemma~\ref{Gabc}, we have
$D_G^n(u)=ac^{2n}F_n\left(\frac{b}{c}\right)$. Note that
$$D_G(v)=D_G\left(\frac{1}{bc}\right)=-\frac{1}{bc}(b^2+c^2),~D_G\left(\frac{1}{bc}(b^2+c^2)\right)=-\frac{1}{bc}(b^2-c^2)^2.$$
Since $D_G(b^2-c^2)=0$, it follows that
\begin{equation*}
D_G^{2k-1}\left(\frac{1}{bc}\right)=-\frac{b^2+c^2}{bc}(b^2-c^2)^{2k-2},~D_G^{2k}\left(\frac{1}{bc}\right)=\frac{1}{bc}(b^2-c^2)^{2k}.
\end{equation*}
In the case when $a=c=1$ and $b=x$, we see that
$$D_G^n(u)|_{a=c=1,b=x}=F_n\left(x\right),~D_G^n(v)|_{a=c=1,b=x}=\widetilde{f}_n(x),$$
where $$\widetilde{f}_n(x)=\left\{
                             \begin{array}{ll}
                               -\frac{1+x^2}{x}(x^2-1)^{2k-2}, & \hbox{if $n=2k-1$;} \\
                              \frac{1}{x}(x^2-1)^{2k}, & \hbox{if $n=2k$.}
                             \end{array}
                           \right.
$$
Applying~\eqref{abc02} and Lemma~\ref{lemma2}, we arrive at $F_{n+1}(x)=(-1)^nx^{n+1}{w_n}$,
where $w_n$ is given by
 \begin{equation*}
\begin{vmatrix}
1&\frac{1}{x}&0&0&\cdots&0&0\\
F_1(x)&\widetilde{f}_1(x)&\binom{1}{1}\frac{1}{x}&0&\cdots&0&0\\
F_2(x)&\widetilde{f}_2(x)&\binom{2}{1}\widetilde{f}_1(x)&\binom{2}{2}\frac{1}{x}&\cdots&0&0\\
\vdots&\vdots&\vdots&\vdots&\ddots&\vdots&\vdots\\
F_{n-1}(x)&\widetilde{f}_{n-1}(x)&\binom{n-1}{1}\widetilde{f}_{n-2}(x)&\binom{n-1}{2}\widetilde{f}_{n-3}(x)&\cdots&\binom{n-1}{n-2}\widetilde{f}_1(x)&\binom{n-1}{n-1}\frac{1}{x}\\
F_n(x)&\widetilde{f}_n(x)&\binom{n}{1}\widetilde{f}_{n-1}(x)&\binom{n}{2}\widetilde{f}_{n-2}(x)&\cdots&\binom{n}{n-2}\widetilde{f}_2(x)&\binom{n}{n-1}\widetilde{f}_1(x)
    \end{vmatrix}.
  \end{equation*}
Multiplying the first column by $x$ and multiplying each of the other columns by $-x$, we get the desired result. This completes the proof.
\end{proof}

By~\eqref{Hn-recu} and Theorem~\ref{Ff02}, we arrive at the following result.
\begin{corollary}
For any $n\geqslant 0$, let
$${f}_n(x)=\left\{
                             \begin{array}{ll}
                              {(1+x^2)}(x^2-1)^{2k-2}, & \hbox{if $n=2k-1$;} \\
                              -(x^2-1)^{2k}, & \hbox{if $n=2k$.}
                             \end{array}
                           \right.
$$
Then the alternating run polynomials of dual Stirling permutations satisfy the recursion
$$F_{n+1}(x)=xF_n(x)+\sum_{r=2}^{n+1}\binom{n}{r-2}f_{n-r+2}(x)F_{r-1}(x).$$
\end{corollary}

\section{Concluding remarks}
In this paper, we develop a general method to deduce determinantal representations of enumerative polynomials.
When we get the recurrence relation of a polynomial and its corresponding grammar, 
the key point in solving this problem is deducing iterated expressions. Some problems may be difficult. 
Along the same lines, it would be interesting to deal with any other Eulerian-type polynomial that was collected by Hwang et al.~\cite{Hwang20}.
\section*{Acknowledgements.}
This paper is dedicated to Professor Fuji Zhang on the occasion of his 90th birthday.
The first author was supported by the National Natural Science Foundation of China (Grant number 12071063)
and Taishan Scholar Foundation of Shandong Province (No. tsqn202211146).
The second author was supported in part by National Natural Science Foundation of China (Grant number 12361072),
2023 Xinjiang Natural Science Foundation General Project, PR China (2023D01A36)
and 2023 Xinjiang Natural Science Foundation For Youths, PR China (2023D01B48).
The fourth author was supported by the National Science and Technology Council
(Grant number: MOST 112-2115-M-017-004).
\bibliographystyle{amsplain}

\end{document}